\documentclass[11pt]{amsart}

\usepackage{latexsym,kotex,fullpage}
\usepackage{tikz}
\usetikzlibrary{arrows}
\usepackage{pstricks-add}
\usepackage{amssymb}
\usepackage{amsfonts}
\usepackage{amsmath}
\usepackage{amsthm}
\usepackage{graphics,pstricks}
\usepackage{comment,youngtab,ytableau}
\usepackage{tabulary,longtable,tabu,multirow,multicol}
\usetikzlibrary{arrows}
\usepackage{pgfplots}
\usepackage{mathrsfs}

\usepackage{tabu,longtable,multirow,multicol,colortbl,xcolor}
\DeclareSymbolFont{legacymaths}{OT1}{cmr}{m}{n}
\SetSymbolFont{legacymaths}{bold}{OT1}{cmr}{bx}{n}
\usepackage[english]{babel}
\usepackage{times}
\usepackage[T1]{fontenc}
\usepackage{color}
\usepackage[colorlinks=true, linkcolor=blue, anchorcolor=blue, citecolor=blue, filecolor=blue, menucolor= blue, urlcolor=red]{hyperref}

\newtheorem{thm}{Theorem}[section]
\newtheorem{lem}[thm]{Lemma}
\newtheorem{cor}[thm]{Corollary}

\theoremstyle{definition}
\newtheorem{defn}[thm]{Definition}
\newtheorem{defn-rem}[thm]{Definition and Remark}

\newtheorem{conj}[thm]{Conjecture}
\newtheorem{rem}[thm]{Remark}

\numberwithin{equation}{section}

\def\X{{\mathbb X}}

\def\Y{{\mathbb Y}}
\def\P{{\mathbb P}}

\def\ds{\displaystyle}

\def\k{\Bbbk}

\usepackage{calrsfs}
\usepackage{calligra}
\usepackage{yfonts}
\DeclareMathAlphabet{\pazocal}{OMS}{zplm}{m}{n}

\newcommand{\Fa}{\mathcal{F}}
\newcommand{\Ga}{\mathcal{G}}
\newcommand{\Ha}{\mathcal{H}}
\newcommand{\Ca}{\mathcal{C}}
\newcommand{\La}{\mathcal{L}}

\newcommand{\Ma}{\mathcal{M}}

\newcommand{\Na}{\mathcal{N}}
\newcommand{\Ta}{\mathcal{T}}

\begin{document}

\title[The Waldschmidt constant of a standard $\k$-configuration in $\P^2$]{The Waldschmidt constant of a standard $\k$-configuration in $\P^2$}
\thanks{Last updated: \today}

\author[M.V. Catalisano]{Maria Virginia Catalisano}
\address[M.V. Catalisano]{Scuola Politecnica\\ Universit\`a di Genova, Genoa, Italy
}
\email{mariavirginia.catalisano@unige.it}

\author[G. Favacchio]{Giuseppe Favacchio}
\address[G. Favacchio]{Dipartimento di Ingegneria, Università degli studi di Palermo, Viale delle Scienze, 90100 Palermo, Italy}
\email{giuseppe.favacchio@unipa.it}

\author[E. Guardo]{Elena Guardo}
\address[E. Guardo]{Dipartimento di Matematica e Informatica\\
Universit\`a di  Catania, Catania, Italy}
\email{guardo@dmi.unict.it}

\author[Y.S. Shin]{Yong-Su Shin${}^{*}$}
\address[Y.S. Shin]{Department of Mathematics,
Sungshin Women's University, Seoul, 02844, Republic of Korea}
\email{ysshin@sungshin.ac.kr }
\thanks{${}^{*}$Corresponding author}


\begin{abstract} A $\k$-configuration of type $(d_1,\dots,d_s)$,  where $1\leqslant d_1 < \ldots < d_s $ are integers,  is a specific set of points in $\P^2$  that has a number of algebraic and  geometric properties. For example,  the graded Betti numbers and Hilbert functions of all $\k$-configurations in $\P^2$ are determined by the type $(d_1,\dots,d_s)$. However  the Waldschmidt constant of a $\k$-configuration in $\P^2$ of the same type may vary. 
In this paper, we find that the Waldschmidt constant of a $\k$-configuration in $\P^2$ of type $(d_1,\dots,d_s)$ with $d_1\ge s\ge 1$ is $s$.  Then we deal with  the Waldschmidt constants of  standard $\k$-configurations in $\P^2$ of type 
$(a)$, $(a,b)$, and $(a,b,c)$ with $a\ge 1$. In particular, we  prove that the Waldschmidt constant of a standard $\k$-configuration in $\P^2$ of type $(1,b,c)$ with $c\ge 2b+2$  does not depend on $c$.
\end{abstract}

\maketitle
\tableofcontents


\section{Introduction}

A set of points $\mathbb{X}$ in $\mathbb P^2$ is called a {\it $\k$-configuration} of
 type $(d_1,\ldots,d_s)$, where  $1\leqslant d_1 < \ldots < d_s $ are integers, when there exists a partition of $\mathbb X =\mathbb X_1\cup \cdots \cup \mathbb X_s$  and $s$ distinct lines $L_1, \ldots, L_s \subseteq \mathbb{P}^2$  such that, for each $i=1,\ldots,s$ we have $|\mathbb{X}_i|=d_i$, $\mathbb{X}_i \subseteq L_i$ and, for $i>1$, $L_{i}\cap (\mathbb{X}_1\cup \cdots \cup \mathbb{X}_{i-1}) =\emptyset$. The last condition forces a point in $\mathbb X$ to belong to the set $\mathbb X_i$ corresponding to the largest index of a line containing~it.

The $\k$-configurations  were introduced in the 1980s by Roberts and Roitman in \cite{RR} and extensively studied in the literature for their several interesting properties, see for instance  \cite{CGS:1,GGSV:2, GHS:1,GHS:2,GMS:1,GS:1}.

In 1995, Harima \cite{H} extended this definition to $\P^3$, and then in 2001 Geramita, Harima, and Shin  \cite{GHS:1,GHS:5}  generalized  the definition to $\P^n$. Moreover, Roberts and Roitman showed that all $\k$-configurations in $\P^2$ of type $(d_1,\dots,d_s)$ have the same Hilbert function, which is encoded in the type. This result was generalized again by Geramita, Harima, and Shin \cite[Corollary 3.7]{GHS:5} to show that all graded Betti numbers of the associated ideal of a $\k$-configuration in $\P^n$ depend on the type only. However, it should be noted that $\k$-configurations in $\P^n$  of the same type can have very different algebraic and geometric properties \cite{CHT,CH}.
In  \cite[Section 3.3]{CGS:1} the authors showed that two different $\k$-configurations of the same type may have different Waldschmidt constants. 
For an easy example, consider the following two $\k$-configurations $\X$ and $\Y$ in $\P^2$ of type $(1,2,3)$. 

\begin{figure}[ht]
\centering
\vskip -1.9cm
\definecolor{ttqqtt}{rgb}{0.2,0,0.2}
\begin{tikzpicture}[line cap=round,line join=round,>=triangle 45,x=1cm,y=1cm]
\clip(-14.13,5.03) rectangle (8.13,10.91);
\draw [line width=1pt] (-11.17,7.13)-- (-7.55,7.69);
\draw [line width=1pt] (-11.8,6.3)-- (-6.5,6.3);
\draw [line width=1pt] (-9.4,8.7)-- (-8.05,8.21);

\draw [line width=1pt] (-2,9)-- (-4.65,6.03);
\draw [line width=1pt] (-3,9)-- (-0.27,6.03);
\draw [line width=1pt] (-4.83,6.45)-- (-0.99,7.71);
\draw [line width=1pt] (-3.97,7.75)-- (-0.01,6.45);
\draw (-9.61,6.11) node[anchor=north west] {$\X$};
\draw (-2.7,6.13) node[anchor=north west] {$\Y$};

\draw (-12.5,6.6) node[anchor=north west] {$\La_3$};
\draw (-11.87,7.45) node[anchor=north west] {$\La_2$};
\draw (-10.1,9.05) node[anchor=north west] {$\La_1$};

 \draw (-5.36,6.3) node[anchor=north west] {$\Ma_1$};
\draw (-5.55,6.7) node[anchor=north west] {$\Ma_2$};
 \draw (-4.74,8.1) node[anchor=north west] {$\Ma_3$};

\begin{scriptsize}
\draw [fill=ttqqtt] (-8.7,8.45) circle (3.5pt);

\draw [fill=ttqqtt] (-10.127111939186168,7.291330749739156) circle (3.5pt);
\draw [fill=ttqqtt] (-8.54475808615293,7.536114771202862) circle (3.5pt);
\draw [fill=ttqqtt] (-10.95,6.3) circle (3.5pt);
\draw [fill=ttqqtt] (-9.11,6.3) circle (3.5pt);
\draw [fill=ttqqtt] (-7.41,6.3) circle (3.5pt);
\draw [fill=ttqqtt] (-2.4925650557620815,8.44795539033457) circle (3.5pt);
\draw [fill=ttqqtt] (-3.3089496389630537,7.532988517841406) circle (3.5pt);
\draw [fill=ttqqtt] (-1.623248453983267,7.5022153510367415) circle (3.5pt);
\draw [fill=ttqqtt] (-2.419420533782162,7.2409713873527295) circle (3.5pt);
\draw [fill=ttqqtt] (-4.045603659897343,6.7073800490961855) circle (3.5pt);
\draw [fill=ttqqtt] (-0.9352633886169351,6.753748082121721) circle (3.5pt);
\end{scriptsize}
\end{tikzpicture}
\vskip -.8cm
\caption{$\k$-configurations $\X$ and $\Y$ in $\P^2$ of type $(1,2,3)$}
\end{figure}
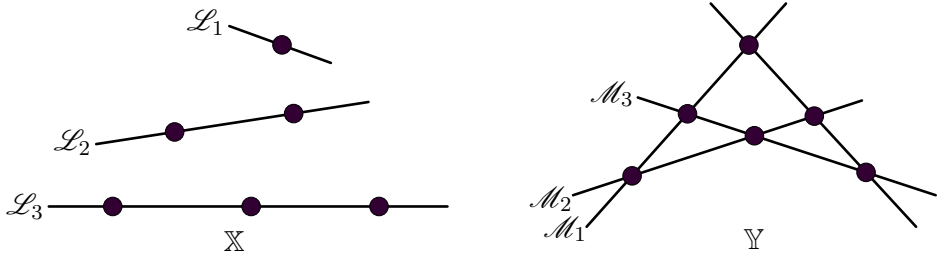
Then the Waldschmidt constants of $\X$ and $\Y$ are different, i.e., 
$$
\widehat\alpha(I_\X)=\frac{7}{3} \quad \text{and} \quad \widehat\alpha(I_\Y)=2,
$$
respectively (see \cite{CGS:1,GHMN}). 

\medskip

The {\em Waldschmidt constant} of a homogeneous ideal $I$ in $R=\k[x_0,x_1,\dots,x_n]$
was introduced in \cite{W} as
$$
\widehat \alpha(I)=\lim_{t\to\infty} \frac{\alpha(I^{(t)})}{t},
$$
where  $I^{(t)}$  is the $t$-th symbolic power of the ideal $I$, defined by $I^{(t)}=\bigcap_{P\in{\rm Ass}(I)}(I^t R_P\cap R)$, and  
$\alpha(I^{(t)})$ is   the least degree among all minimal homogeneous generators of $I^{(t)}$.
 In   \cite[Lemma 2.3.1]{BH}   it was proved that this limit exists. 
 
 Note that if $I_\X$ is the  ideal defining a set of distinct points $\X=\{P_1,\dots,P_s\}$   in $\P^n$ and $I_{P_i}$ is the  ideal of the point $P_i$, then the  $t$-th symbolic power of   $I_\X$ is  $I_{\X}^{(t)} =I^t_{P_1}\cap \cdots \cap I^t_{P_s}$ ,
 that is, $I_{\X}^{(t)}$ defines a
  {\em homogeneous set of fat points supported at $\X$},  
 denoted by $t\X$.
If $I_\X$ is the  ideal of a set of  points $\X$, instead of   \lq\lq Waldschmidt constant of $I_\X$\rq\rq , we
simply write \lq\lq Waldschmidt constant of $\X$\rq\rq.

A prolific line of research involves the study of the Waldschmidt constant of zero dimensional schemes in $\mathbb P^n$, see \cite{BCG^+,   CHHV, DFMS, DHNSST, DHST, DHST:2, HM, HMF, MSS, NCT} just to cite some papers. 


 In this paper we  show  many other cases where $\k$-configurations in $\P^2$ of the same type have different Waldschmidt constants,  and  we extend some results  found in \cite{CGS:1}.
In particular we focus on the so called 
 {\it standard $\k$-configurations} in $\mathbb P^2$, see Definition \ref{standardk},  
   and we find the Waldschmidt constants of all standard $\k$-configurations of type   $(a)$,  $(a, b)$ and $(a, b, c)$, except for type $(2, 3, 5)$, as summarized in Table \ref{T:20210916-101}.

The paper is structured as follows.

 In Section~\ref{preliminari} we recall some definitions and useful tools, in particular we prove,
in a more general context,  the existence of irreducible curves (see Lemma \ref{T:20210527-109}).
 In Section 
\ref{M20210723-1}  we describe a  method  to find the Waldschmidt constant of a set $\X$ of points, that works 
in particular when $\X$ is supported on some lines. 
 In Section 
\ref{d1ds}  we consider particular schemes with support on lines,  when the number of points on each line is bigger than the number of lines.  As an application, we find the Waldschmidt constants of  standard $\k$-configurations of type $(a)$ and, for $a>1$, of type  $(a,b)$. To complete the case $(a,b)$, we recall the result in {\cite[Proposition 3.3]{DHST:2}}.
In Section~\ref{S1bc}, we find the Waldschmidt constants of standard $\k$-configurations of type $(1,b,c)$. 
 In Section~\ref{Sabc}, we find the Waldschmidt constants of standard $\k$-configurations of type $(a,b,c)$, with $a >1$, except the type $(2,3,5)$.

To lighten the reading load,  the proofs of some theorems of Section 5, that are very similar to the  proofs of other theorems in the same section, can be found in the Appendix, where an interested reader will find all the details. 

\begin{table}[h] 
\renewcommand{\arraystretch}{2.1}
\[
\begin{array}[t]{|c|c|c|c|cccccccccccccccccc}
\hline
\text{The type of $\X$} & \text{Note} &\widehat\alpha(I_{\X})  & \text{From}    \\ \hline
(a)  & & 1 &  \text{Corollary~\ref{M20210802-1}}  \\  \hline
 (1,b)  & &  \dfrac{2b-1}{b} &  \text{Remark~\ref{Sa-ab}}  \\  \hline
 (a,b)  & \text{ $a\geq 2$   }& 2 &  \text{Corollary~\ref{M20210802-1}}  \\  \hline
(1,b,b+1)  & \text{$b$ even, $b\geq 4$   }&  \dfrac{9b-4}{3b} &  \text{Theorem~\ref{T:20210606-402}}  \\  \hline
(1,b,c)  & \text{$c$ even, $c\leq 2b-4$   }&  \dfrac{6b+3c-4}{2b+c} &   \text{Theorem~\ref{MMcpari}} \\  \hline
(1,b,c)  & \text{$c$ odd, $b+1 < c\leq 2b-3$   }&  \dfrac{6b+3c-7}{2b+c-1} &   \text{Theorem~\ref{MMcdispari}} \\  \hline
(1,b,2b-2)  & &  \dfrac{6b^2-14b+6}{2b^2-4b+1} &   \text{Theorem~\ref{M2b-2}} \\  \hline
(1,b,2b-1)  & &  \dfrac{6b^2-8b+1}{2b^2-2b} &   \text{Theorem~\ref{T:20210717}} \\  \hline

(1,b,2b)  & &  \dfrac{6b-5}{2b-1}   &  \text{Theorem~\ref{T:20210428-203}}  \\  \hline
(1,b,2b+1)  & &  \dfrac{6b^2-2b-3}{2b^2-1}   & \text{Theorem~\ref{T:20210618-203}}   \\  \hline
(1,b,c)  &\text{$c \geq 2b+2$   }&   \dfrac{3b-1}{b}  & \text{Theorem~\ref{T:20210617-419}}   \\  \hline
(2,3,4)  & &   \dfrac{17}{6}  &  \text{Theorem~\ref{M20210728-234}} \\  \hline
(2,3,5)  & &   \dfrac{17}{6}\leq \widehat\alpha(I_{\X})\le \dfrac{71}{24} &  \text{Remark~\ref{rem235}} \\  \hline
(2,3,c)  & c\ge 6 &   3  &  \text{Theorem~\ref{M20210730-2bc}} \\  \hline
(2,b,c)  & b\ge 4 &   3  &  \text{Theorem~\ref{M20210730-2bc}}\\  \hline
(a,b,c)  & a\ge 3 &   3  &  \text{Theorem~\ref{M20210730-abc}} \\  \hline
\end{array}
\]

\caption{The Waldschmidt constant of  standard $\k$-configurations of type $(a)$, $(a,b)$, $(a,b,c)$}
\label{T:20210916-101}
\end{table}

\medskip

{\bf Acknowledgements}. 
The first author wishes  to thank the hospitality of Universit\`a di Catania during an early stage of this work.  
The first author was  supported  by  Universit\`a degli Studi di Genova through the FRA (Fondi per la Ricerca di Ateneo) 2018.
The second author was partially supported by MIUR grant Dipartimenti di Eccellenza 2018-2022 (E11G18000350001).
The third author was supported by Universit\`a di Catania, Progetto Piaceri 2020/22, linea  Intervento 2. 
The first, second and third authors  are members of GNSAGA of INDAM. 
The last author was supported by a grant from Sungshin Women's University.
Our results were inspired by computations using CoCoa \cite {cocoa} and Macaulay2 \cite {M2}.


\section{Preliminaries} \label{preliminari}

We  work with sets of points in $\mathbb P^2=\mathbb P^2_{\k}$, where $\k$ is an algebraic closed field of characteristic zero. The ideal defining  a set of points $\X$ will be denoted by  $I_{\X}$.  We will refer to  $[I_{\X}]_d$  as the linear system of all the plane curves of degree $d$ containing $\X$, since this is, from a geometrical point of view, what the forms in $[I_{\X}]_d$ correspond to, and we simply write $\dim[I_{\X}]_d$ instead  of $\dim_\k [I_{\X}]_d$.

We recall the definition of  the Waldschmidt constant for an ideal (see for instance  \cite{BH}). 
\begin{defn}\label{d.WC}Let $\X$ be a subscheme of $\mathbb P^n$ defined by  the homogeneous ideal $I_{\X}\subseteq \k[\mathbb P^n]$. We denote by  $\alpha(I_{\X})$, the initial degree of $I_{\X}$, i.e.,  the least degree of  nonzero elements in $I_{\X}$. 
The Waldschmidt constant of  $I_{\X}$ is the following limit
    $$\widehat \alpha(I_{\X}) = \lim_{t\rightarrow \infty} \frac{\alpha(I_{\X}^{(t)})}{t} ,
$$
where $I_{\X}^{(t)}$ is the $t$-th symbolic power of $I_{\X}$.
\end{defn}

In this paper we willl work with   special sets of simple  distinct  points in $\mathbb{P}^2$. In this case 
$I_{\X}^{(t)}= I_{t\X}$, and we have the following useful lemma.

\begin{lem}\label{L:20210722-1} 
Let $\X$ be a set of simple  distinct  points in $\mathbb{P}^2$, and let
$I_\X$ be its ideal. Let  $\mu$ and $d$  be positive integers such that
the initial degree of the scheme 
$m\mu \X$ is $md$ for each integer $m >0  $. Then the Waldschmidt constant of $I_\X$ is 
$$ \widehat \alpha(I_\X) = \frac {d}{\mu}.
$$
\begin {proof} Since, by definition,
$\widehat \alpha(I_\X) = \lim_{t\rightarrow \infty} \frac{\alpha(I_\X^{(t)})}{t},  $
if we let $t=m\mu$, we have  $\alpha(I_\X^{(t)})=\alpha( I_{m\mu \X})=md$, and so
$$ \widehat \alpha(I_\X) = \frac {md}{m\mu}=\frac {d}{\mu}.
$$ \end{proof}
\end{lem}

We now recall the definitions of $\k$-configurations and standard  $\k$-configurations. 
\begin{defn} [\cite{GHS:1,GHS:2,RR}]\label{kconfig}
  
 	 Let $1\leqslant d_1 < \ldots < d_s $ be integers and let $L_1, \ldots, L_s \subseteq \mathbb{P}^2$ be distinct lines. A \emph{$\k$-configuration of points in $\mathbb{P}^2$ of
      type $(d_1,\ldots,d_s)$} is a finite set
  $\mathbb{X}$ of points in $\mathbb{P}^2$  such
  that:
 
  \begin{enumerate}
  \item $\mathbb{X} = \bigcup_{i=1}^s \mathbb{X}_i$, where the $\mathbb{X}_i$ are   subsets of $\mathbb{X}$ ;
  \item $|\mathbb{X}_i| = d_i$ and
    $\mathbb{X}_i \subseteq L_i$ for each $i=1,\ldots,s$;
  \item $L_i$ ($1< i \leqslant s$) does not contain any
    points of $\mathbb{X}_j$ for all $j<i$.
      \end{enumerate}

\end{defn}

In analogy of \cite[Section 4]{GHS:1} in $\P^3$ and \cite[Section 4]{GHS:2} in $\P^n$, here we give an explicit definition of  standard $\k$-configurations in $\P^2$, which are special $\k$-configurations of points in $\P^2$ whose coordinates are  integer values.

\begin{defn} \label{standardk}
	Let $\k[x_0,x_1,x_2]$ be the  homogeneous ring for $\P^2$, and let $(d_1,\dots , d_s)$ be the type of a $\k$-configuration in $\P^2$. We construct a set of points  which realizes this type, and whose points   are located in the following   lines $ L_i$, where	
	\[  L_1=\{ x_2=(s-1)x_0\},  L_{2}=\{ x_2=(s-2)x_0\}, \quad \dots \quad, L_s= \{ x_2=0\}.\]
		 On each of these lines $ L_i$ we place $d_{i}$ points  as follows
			\begin{center}
			\begin{tabular}{lllll}
			&$d_1$ points on $L_1$  with coordinates $[ 1:j:s − 1 ]$ & $0 \leq j \leq d_1 − 1,$\\
			 &$d_2$ points	on $L_2$  with coordinates $[1:j:s-2 ]$  & $0\leq j \leq d_2-1$, \\
				& $\vdots$\\
			& $d_{s}$ points on $L_s$   with coordinates $[1:j:0 ]$ & $0\leq j \leq d_{s}-1$. \\

		\end{tabular}
		\end{center}
	If   $1\leq d_1<\dots <d_s$, we call  the $\k$-configuration of points in $\P^2$ constructed  as above a {\em standard $\k$-configuration} of type $(d_1,\dots , d_s).$

	\end{defn}

We conclude this section with two lemmas,  that are 
 key tools for the proofs in this paper. 
 
 The first one is a technical lemma from our previous paper 
\cite{CGS:1}, and it is an application of Bezout's Theorem.

The second lemma is 
useful to compute the Waldschmidt constants of all the standard $\k$-configurations  from  type $(1,b,2b-2)$ to  $(1,b,2b+1)$, since for those cases we need the existence of
 irreducible curves.

 \begin{lem} \label{L:20210429-201} Let  $m_1, \ldots, m_s$ and $d $ be positive integers and let $P_1, \ldots,P_s $  be $s$ points   lying on a line $\mathcal L$ with $s>1$. Let $\X$  be the scheme $m_1P_1+ \cdots+m_sP_s $. Set
\begin {equation} \label {mu}  \mu =\left \lceil \frac {m_1+\cdots +m_s -d }{s-1} \right \rceil,
\end {equation}
and assume  $[I_\X]_d \neq \{0\}$. Then 

\begin{enumerate}
\item [(i)]  $\mu \leq d$;

\item [(ii)] 
the line $\mathcal L$ is a fixed component of multiplicity at least $\mu$ for the plane curves of degree $d$  defined by the forms of the ideal $[I_\X]_d$. 
\end{enumerate}
\end{lem}

\begin{proof}  (i) Since $[I_\X]_d \neq \{0\}$, then $d \geq m_i$ for any $i$, hence
$$\mu =\left \lceil \frac {m_1+\cdots +m_s -d }{s-1} \right \rceil \leq \left \lceil \frac {sd -d }{s-1} \right \rceil =d;
$$
(ii)  follows from
{\cite[Lemma 2.5]{CGS:1}}.
\end{proof}
 \begin{rem}
 Note that the condition $\mu \leq d$  follows from the hypothesis $[I_\X]_d \neq \{0\}$. Hence   we should not have to put that condition  in  {\cite[Lemma 2.5]{CGS:1}}'s  hypotheses.
 \end{rem}

\begin{lem} \label{T:20210527-109} 
Let $L$, $M$ be two distinct  lines, and let $b$ be a positive integer. Let  $P_1,\dots,P_b $,  $Q_1,\dots,Q_b$, $R$ be distinct points such that  $R \not\in L \cup M$, and,  for any $1 \leq i  \leq b$,  $P_i \in L$,  $Q_i \in M$, and the point  $L \cap M \not\in \{P_1,\dots,P_b , Q_1,\dots,Q_b\}$. Moreover $R$, $P_i$, $Q_j$ do not lie on a line, for any $i$ and $j$. Then
\begin{enumerate}
\item [(i)] the scheme $\X= P_1+\cdots +P_b + Q_1+\cdots +Q_b +(b-1)R$ gives independent conditions to the curve of degree $b$ (see Figure \ref{FIG:20230726});
\item [(ii)] the only curve of degree $b$   in $[I_{\X}]_{b}$ is irreducible.
\end{enumerate}
\end{lem}

\begin{figure}[ht] 
\centering
\vskip -3.3cm
\psset{xunit=.8cm,yunit=.8cm,algebraic=true,dimen=middle,dotstyle=o,dotsize=7pt 0,linewidth=1pt,arrowsize=3pt 2,arrowinset=0.25}
\begin{pspicture*}(-7.14,-9.2)(12.4,8.6)

\psline[linewidth=1pt](0.04,4.38)(7.9,4.38)
\psline[linewidth=1pt](0.34,2.32)(6.,3.5)

\rput[tl](-0.33,3.55){\tiny$R$}
\rput[tl](-0.48,4.53){\tiny$L$}
\rput[tl](-0.16,2.45){\tiny$M$}
\rput[tl](2.79,4.59){\Huge$\cdots$}
\rput[tl](2.55,3.02){\Huge$\cdot$}
\rput[tl](3.05,3.13){\Huge$\cdot$}
\rput[tl](3.55,3.24){\Huge$\cdot$}

\begin{scriptsize}
\psdots[dotstyle=*](0.2,3.42)
\psdots[dotstyle=*](1,4.4)
\psdots[dotstyle=*](2.0,4.38)
\psdots[dotstyle=*](5.0,4.38)
\psdots[dotstyle=*](6.0,4.38)
\psdots[dotstyle=*](7.0,4.38)
\psdots[dotstyle=*](1.2713797048392632,2.514939938222171)
\psdots[dotstyle=*](2.1432845212218283,2.6974316439766617)
\psdots[dotstyle=*](4.268294245509667,3.142201121153186)
\psdots[dotstyle=*](5.087365290012585,3.313634595584029)
\end{scriptsize}
\end{pspicture*}\vskip -9.4cm
\caption{The scheme $\X$ }\label{FIG:20230726}
\end{figure}

\begin{proof} 
(i) It is well known that the fat point $(b-1)R$ gives independent conditions to the curve of degree $b$. Consider the following curve $\Ga_i$ of degree $b$ 
$$\Ga_i=L + N_1+ \cdots + N_{i-1} +  N_{i+1} \cdots + N_b ,$$
where $N_j$ is the line $RQ_j$, $j\neq i$, so that $\Ga_i$ contains  the scheme $\X - Q_i$, but it does not contain $Q_i$. Analogously we can construct a curve of degree $b$ passing through $\X - P_i$,  that does not contain $P_i$. Hence $\{P_1, \ldots ,P_b , Q_1, \ldots ,Q_b\}$ gives independent 
conditions to the curves defined by the linear system $[I_{(b-1)R}]_{b}$, and thus (i) follows.

(ii)  Note that since
$$
\binom{b+2}{2}-\bigg(b+b+\binom{b}{2}\bigg)=1,
$$
then from (i) there exists only a curve of degree $b$ through $\X$, say $\Ca$.
Now we prove by induction on $b$ that  the curve  $\Ca$ is irreducible.  Obvious for $b=1$, assume  $b>1$. Assume that 
$$\Ca = \Ca_1+\cdots+\Ca_r,$$
where $r>1$ and the $\Ca_i$ are the irreducible components of $\Ca$.  Let $b_i = \deg \Ca _i$, and let $m_i$ be the multiplicity of $\Ca _i$ at $R$.

 Note that if  $b_i =1$, i.e., $\Ca_i$ is a line, then  $m_i \leq 1$;    if  $b_i >1$, since $\Ca_i$ is irreducible,  then $m_i \leq b_i-1$.

If for each $i$ we have $b_i >1$,  then
$$b-1 \leq m_1+ \cdots +m_r \leq( b_1-1)+ \cdots + ( b_r-1) = b-r , $$
hence $r\leq1$, and we get a contradiction.

Otherwise, without loss of generality, we can assume that $b_1=1$, that is, $\Ca_1$ is a line.

If $R \not \in \Ca_1$, then $\Ca_1$ contains at most $b$ simple points of $\X$. So since the curve $\Ha=\Ca_2+\cdots+\Ca_r$ has degree $b-1$, and contains the fat point $(b-1)R$,  then it is union of $b-1$ lines through $R$. Moreover, since  each line through $R$ contains at most one point of $\X- (b-1)R$, then $\Ha$ cannot contains $\X- (b-1)R-\Ca_1$.
Hence $R  \in \Ca_1$ and so  $\Ca_1$ contains at most one other point of $\X$. Hence $\Ha = \Ca_2+\cdots+\Ca_r$ is a curve of degree $b-1$ through $\X-\Ca_1$, that is, through $(b-2)R$ and at least $2b-1$  points in the set
$\{P_1, \ldots ,P_b , Q_1, \ldots ,Q_b\}$. We may assume that $\Ha$ contains 
$P_1+ \cdots +P_b +Q_1+\cdots +Q_{b-1}$.
By the inductive hypothesis, the only   curve  of degree  $b-1$ through 
$(b-2)R+P_1+ \cdots +P_{b-1} +Q_1+\cdots +Q_{b-1}$ is irreducible. Hence $\Ha$ has to be that curve. But 
$P_b \in \Ha $, so, by Bezout's Theorem, $L$ is a component of $\Ha $, hence, since $\Ha$ is irreducible, we get $L=\Ha $. It follows that $Q_1 \in L$, a contradiction.
\end{proof}

\section{Method} \label{M20210723-1}

In this section we describe the main method  to find the Waldschmidt constant of a set $\X$ of points in $\P^2$. 
Our computation is structured as follows.
\begin{enumerate}
\item[\it Step 1.] We look for a  curve $\Fa$ of degree $d$, which contains each point of   $\X$ with  multiplicity exactly $\mu$, so that, for each $m>0$,  $m\Fa$ is a curve in the linear system $\big[I_{m\mu\X}\big]_{md}$ and  so $\big[I_{m\mu\X}\big]_{md}\neq \{0\}.$

\item[\it Step 2.] 
We show that $\big[I_{m\mu\X}\big]_{md-1}= \{0\},$ for each $m\ge 1$
and we prove it  by contradiction. 
For this purpose we define
$$\bar m =\min \{ m | [I_{m \mu \X}]_{m d-1}\neq \{0\} \}.
$$
We prove, mostly directly, that  $\bar m \neq 1$. For $\bar m>1$,   applying Lemma \ref {L:20210429-201} several times,
we show that $\Fa$ is a fixed component for the linear system
$\big[I_{\bar m\mu\X}\big]_{\bar md-1}.$ 
Thus, by removing $\Fa$, we get
$$\dim \big[I_{\bar m\mu\X}\big]_{\bar md-1}= \dim \big[I_{\bar m\mu\X- \Fa}\big]_{\bar md-1-d}$$
 and, since $\Fa$ contains each point of   $\X$ with  multiplicity exactly $\mu$, we have
 $$\big[I_{\bar m\mu\X- \Fa}\big]_{\bar md-1-d}=\big[I_{(\bar m-1)\mu\X}\big]_{(\bar m-1)d-1}$$
 and the contradiction comes from the minimality of $\bar m$.

\item[\it Step 3.] Since the initial degree of $\big[I_{m\mu\X}\big]$ is $md$, then, by Lemma \ref {L:20210722-1} we have 
 $$
\widehat \alpha (I_{\X}) = \frac{d}{\mu}.
$$
\end{enumerate}

Note that if $\X$ is  a standard $\k$-configuration, then the curve $\Fa$  strictly depends on the type of $\X$. In certain cases it is a union of lines, and in other cases it has irreducible components of higher degrees.

\section{Waldschmidt constants of  $\k$-configurations   of type 
$(d_1,\dots, d_s)$ with $ d_1 \geq s$} \label{d1ds}

 In the next lemma we compute the Waldschmidt constant of a set of points $\X$ contained in $s$ lines, where each line contains at least $s$ points of $\X$ and no two lines meet in $\X$.

The following lemma will be useful for computing the Waldschmidt constants of  both a $\k$-configuration  of type $(d_1,\dots, d_s)$ and a  standard $\k$-configuration of the same type $(d_1,\dots, d_s)$, when $ d_1 \geq s$.

\begin{lem}\label{L:20210529-205} Let $ s$ be a positive integer, and let $L_1, \ldots, L_s $ be distinct lines.
Let $\X_i$ be a finite set of $d_i$ points on the line $L_i$ ($1\le i\le s$), and let ${\X} = \bigcup_{i=1}^s {\X}_i$. If $d_i \geq s$, for each $1 \leq i \leq s$, and any  intersection point of two lines $L_i$ and $L_j$, for $ i\ne j$, is not contained in $\X$, then the Waldschmidt constant of $\X$ is 
$$
\widehat\alpha (I_\X)=s.
$$ 
\end{lem}       

\begin{proof} For $s=1$, it is immediate. So we assume $s>1$.

Let  $m$ be a positive integer. The curve
$
\Fa=L_1+\cdots+L_s
$
has degree $s$ and passes through  the points of $\X$ with multiplicity  $1$, hence
$$
m\Fa\in [I_{m\X}]_{m s}.
$$
Now we prove that for each $m >0$,
$$
[I_{m \X}]_{m s-1}=\{0\},
$$
so the initial degree of  $I_{m \X}$ will be $ms$ and the conclusion will follow from Lemma \ref {L:20210722-1}.

Assume that for some $m$, $[I_{m \X}]_{m s-1}\neq\{0\}$, and note that if $[I_{m \X}]_{m s-1}\neq\{0\}$, then by Lemma \ref {L:20210429-201}, since $d_i\ge s\ge 2$, each $L_i$ is a fixed component of multiplicity at least 
\begin {equation} \left \lceil \frac {md_i -(ms-1) }{d_i-1} \right \rceil \geq 1,
\end {equation}
for the plane curves of the linear system $[I_{m \X}]_{m s-1}$, hence $\Fa$  is a fixed component for the  curves defined by this linear system.

Set 
\begin {equation} \label {barm1}  \bar m =\min \{ m | [I_{m \X}]_{m s-1}\neq \{0\} \}.
\end {equation}
First observe that $\bar m \neq 1$. In fact, for 
$m =1$, since $\deg\Fa=s$, then $[I_{\X}]_{s-1}=\{0\}$.
By removing $\Fa$
from the  curves  of  the linear system $[I_{\bar m \X}]_{\bar m s-1}$,  since  any  intersection point of two lines $L_i$ and $L_j$ is not contained in $\X$, we get
$$
\dim [I_{\bar m \X}]_{\bar m s-1}=\dim [I_{\bar m \X-\Fa}]_{(\bar m s-1)-s}=\dim [I_{(\bar m-1) \X}]_{(\bar m-1) s-1},
$$
and by \eqref {barm1} this is zero, a contradiction.
\end{proof}

\begin{cor} \label{M20210802-1} Let  $\X$ be  a standard $\k$-configuration of type $(d_1,\dots, d_s)$ with $ d_1 \geq s$, then the Waldschmidt constant of $\X$ is
$$
\widehat \alpha(I_\X)=s. 
$$
\end{cor} 

\begin{proof} It follows from the previous lemma.
\end{proof}

\begin{cor} \label{P:20210529-206} With notation as in Definition~\ref{kconfig}, if $\X$ is a $\k$-configuration  of type $(d_1,\dots, d_s)$ with $ d_1\geq s$, then the Waldschmidt constant of $\X$ is
$$
\widehat \alpha(I_\X)=s. 
$$
\end{cor} 

\begin{proof} 
Let
$\Fa=L_1+\cdots+L_s,$ thus 
 $ m \Fa \in [I_{m \X}]_{m s}. $
Hence
\begin{equation*} \label{EQ:20210529-202} 
\widehat \alpha (I_\X)\le s.
\end{equation*} 
Now let $\X'$ be the subset of $\X$  that we get after we remove the possible points of $\X$ in the intersections $L_i \cap L_j$, for   $ i \neq j$. Let $\X'_i =\X' \cap L_i $.  Recalling Definition~\ref{kconfig}  it is easy to show that by  Lemma \ref{L:20210529-205} we have
$$
\widehat \alpha(I_{\X'})=s. 
$$
Since $\X' \subseteq \X$, we have  $\widehat \alpha(I_{\X'}) \leq \widehat \alpha(I_{\X})$.
Thus, the conclusion follows from
$$
s=\widehat \alpha (I_{\X'})\le \widehat \alpha (I_\X)\le s. 
$$
\end{proof}

 \begin{rem} \label{Sa-ab}From Corollary \ref{M20210802-1}, we immediately get that the Waldschmidt constant of  a standard $\k$-configuration of type $(d_1)$ is 1, and  of type  $(d_1,d_2)$ with $d_1 \geq 2$ is 2. 
 For the case $(1,d_2)$ see  {\cite[Proposition 3.3]{DHST:2}}, where it is proved that  
 if  $\X$ is a { standard} $\k$-configuration of type  $(1,d_2)$, then 
$
\widehat\alpha(I_\X)=\frac{2d_2-1}{d_2}.
$
   \end{rem}

\section{Waldschmidt constants of  standard $\k$-configurations  of type $(1,b,c)$} \label{S1bc}

 In this section we compute the Waldschmidt constant of a standard $\k$-configuration $\X$ of type $(1,b,c)$ as in Definition \ref{standardk}, for any values of $b$ and $c$.

It is interesting to note that  the Waldschmidt constant  stabilizes at $c= 2b+2$, that is, 
$$
\widehat\alpha(I_\X)=\frac{3b-1}{b} \  \hbox{ for } \ c\geq 2b+2
$$
(see Theorem \ref{T:20210617-419}).
One could expect that, for each fixed $b$,  the Waldschmidt constant strictly increases with $c$  until  $c= 2b+2$. But this is not always the case, as shown in Corollary \ref{ME20210909}, since for $c\leq 2b-3$  it behaves in a similar way as a step function.

We fix the notation of this section,  summarized in Figure \ref{figuragenerale1}, that will be used in the proofs.

\begin{figure}[ht] 
\centering
\vskip -5.5cm
\psset{xunit=1cm,yunit=1cm,algebraic=true,dimen=middle,dotstyle=o,dotsize=7pt 0,linewidth=.5pt,arrowsize=3pt 2,arrowinset=0.25}
\begin{pspicture*}(-10,-8.8)(5,8.8)

\psline[linewidth=1pt](-6.7,1)(1.7,1)

\psline[linewidth=1pt](-6.7,2)(0.7,2)

\psline[linewidth=1pt](-6,3.5)(-6,0.65)

\psline[linewidth=1pt](-6.23,3.46)(-4.81,0.65)

\psline[linewidth=1pt](-6.41,3.37)(-3.65,0.66)

\psline[linewidth=1pt](-6.41,3.25)(-2.45,0.66)

\psline[linewidth=1pt](-6.51,3.2)(-1.25,0.66)

\psline[linewidth=1pt](-6.51,3.1)(1.25,1.66)

\rput[tl](-0.65,2.17){\Huge $\cdots$}
\rput[tl](0.0,1.17){\Huge$\cdots$}

\rput[tl](-7.15,1.1){\tiny$\La_1$}
\rput[tl](-7.15,2.1){\tiny$\La_2$}

\rput[tl](-6.40,2.9){\tiny$R$}
\rput[tl](-6.40,2.3){\tiny$Q_1$}
\rput[tl](-5.48,2.3){\tiny$Q_2$}
\rput[tl](-0.00,2.3){\tiny$Q_b$}

\rput[tl](-6.40,1.3){\tiny$P_1$}
\rput[tl](-5.48,1.3){\tiny$P_2$}
\rput[tl](-4.55,1.3){\tiny$P_3$}
\rput[tl](-3.65,1.3){\tiny$P_4$}
\rput[tl](0.7,1.3){\tiny$P_c$}

\rput[tl](-6.27,0.6){\tiny$\Ma_1$}
\rput[tl](-4.9,0.6){\tiny$\Na_1$}
\rput[tl](-3.7,0.6){\tiny$\Ma_2$}
\rput[tl](-2.6,0.6){\tiny$\Na_2$}
\rput[tl](-1.5,0.6){\tiny$\Ma_3$}
\rput[tl](1.3,1.7){\tiny$\Ta_i$}

\begin{scriptsize}
\psdots[dotstyle=*](-6,3)
\psdots[dotstyle=*](-6,2)
\psdots[dotstyle=*](-5,2)
\psdots[dotstyle=*](-4,2)
\psdots[dotstyle=*](-3,2)
\psdots[dotstyle=*](-6,1)
\psdots[dotstyle=*](-5,1)
\psdots[dotstyle=*](-4,1)
\psdots[dotstyle=*](-3,1)
\psdots[dotstyle=*](-2,1)
\end{scriptsize}
\end{pspicture*}
\vskip -9.4cm
\caption{A standard $\k$-configuration of type $(1,b,c)$}\label{figuragenerale1}
\end{figure}
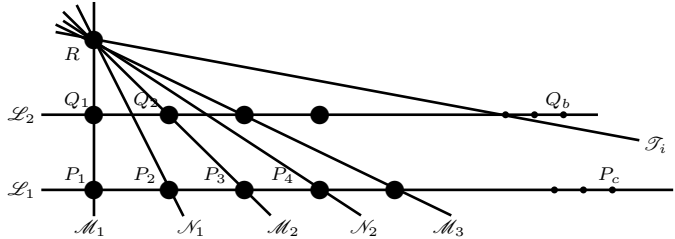

Let  $P_i=[1: i-1:0]$, for $1\le i\le c$,  $Q_i=[1:i-1:1]$, for $1\le i\le b$, and $R=[1:0:2]$ be the points of $\X$ (see Definition \ref{standardk}). 

We denote by  
$$
\begin{array}{llllllllllllllllll}
\La_1 & \text{be the line through  $P_1,P_2,\dots,P_{c}$;} \\
\La_2 & \text{be the line through  $Q_1,Q_2,\dots,Q_{b}$;} \\
\Ma_1 & \text{be the line through  $P_1,Q_1,R$;} \\
\Ma_2 & \text{be the line through  $P_3,Q_2,R$;} \\
   & \hskip 2 cm \vdots \\
  \Ma_i& \text{be the line through  $P_{2i-1},Q_i,R$,
       for $i\leq b$ and $2i\leq c+1$;}\\
\Na_1 & \text{be the line through  $P_2,R$;} \\
\Na_2 & \text{be the line through  $P_4,R$;} \\
      & \hskip 2 cm \vdots \\
 \Na_i  & \text{be the line through  $P_{2i},R$, \ for  $2i\leq c$;}\\

 \Ta_i &  \text{be the line through  $Q_{i},R$,  for $i\leq b$ and $2i\geq c+2$.}\\
\end{array}
$$
Note that each line $\Ma_i$ contains three points of $\X$, whereas the lines $\Na_i$ and  $\Ta_i$ contain two points of $\X$.

\begin{thm}\label{MMcpari} Let $\X$ be a { standard} $\k$-configuration of type  $(1,b,c)$. If $c$ is even and $c\le 2b-4$, then 
$$
\widehat\alpha(I_\X)=\frac{6b+3c-4}{2b+c}.
$$
\end{thm}
 
\begin{proof}
Define
$$
\Fa=\frac{2b+c-2}{2}\La_1 +\frac{2b+c-2}{2}\La_2 +\Ma_1+\cdots+\Ma_{\frac{c}{2}}+\Na_1+\cdots+\Na_{\frac{c}{2}}+\Ta_{\frac{c+2}{2}}+  \cdots+\Ta_{b}.
$$
$\Fa$ is the union of   $\frac{6b+3c-4}{2}$ lines, and $\Fa$  contains each point of   $\X$ with  multiplicity  exactly $\frac{2b+c}{2}$.
Hence, for $m >0$,
$$
m \Fa\in \big[I_{\frac{2b+c}{2}m \X}\big]_{\frac{6b+3c-4}{2}m }.
$$
Now we prove that for each $m >0$,
$$
\dim \big[I_{\frac{2b+c}{2}m \X}\big]_{\frac{6b+3c-4}{2}m -1 }=0,
$$
and the conclusion will  follow from Lemma \ref {L:20210722-1}.

Assume that for some $m$, $\big[I_{\frac{2b+c}{2}m \X}\big]_{\frac{6b+3c-4}{2}m -1 }\neq\{0\}$, thus by Lemma \ref {L:20210429-201},  by recalling that $c >b$, we get that $\La_1$ is a fixed component of multiplicity at least 
\begin{equation}\label{MM1} 
\left\lceil \frac{\frac{2b+c}{2}cm-\frac{6b+3c-4}{2}m +1}{c-1}\right\rceil=
\left\lceil \frac{((2b+c-6)(c-1)+4c-4b-2)m +2}{2(c-1)}\right\rceil \ge {\frac{2b+c-6}{2} m} 
\end{equation}
for the plane curves of the linear system
$\big[I_{\frac{2b+c}{2}m \X}\big]_{\frac{6b+3c-4}{2}m -1 }$.

By removing  $ {\frac{2b+c-6}{2} m} \La_1$ from those curves, we get 
$$\dim [I_{\frac{2b+c}{2}m \X}]_{\frac{6b+3c-4}{2}m -1 }=
\dim [I_{\frac{2b+c}{2}m \X - \frac{2b+c-6}{2}m \La_1}]_{\frac{6b+3c-4}{2}m -1 -{\frac{2b+c-6}{2} m}}.
$$
If the dimension above is zero, we get a contradiction and we are done. If it is different from zero, 
by Lemma \ref {L:20210429-201}, by observing that  
$\frac{6b+3c-4}{2}m -1 -{\frac{2b+c-6}{2} m}=(2b+c+1)m -1$, we get that $\La_2$ is a fixed component of multiplicity at least 
\begin{equation}\label{MM2} 
\left\lceil \frac{\frac{2b+c}{2} bm - ( 2b+c+1 )m +1}{b-1}\right\rceil=
\left\lceil \frac{((2b+c-6)(b-1)+4b-c-8)m +2}{2(b-1)}\right\rceil \ge {\frac{2b+c-6}{2}m} ,
\end{equation}
for the plane curves  of  the linear system
$[I_{\frac{2b+c}{2}m \X - \frac{2b+c-6}{2}m \La_1}]_{(2b+c+1)m -1}.
$
By removing $\frac{2b+c-6}{2}m \La_2$ from those curves, we get 
\begin{equation}\label{MYS1} 
\dim [I_{\frac{2b+c}{2}m \X - \frac{2b+c-6}{2}m \La_1}]_{(2b+c+1)m -1}=
\dim [I_{\frac{2b+c}{2}m \X - \frac{2b+c-6}{2}m \La_1- \frac{2b+c-6}{2}m \La_2}]_{(2b+c+1)m -1 -{\frac{2b+c-6}{2} m}},
\end{equation}
where 
$$\frac{2b+c}{2}m \X - \frac{2b+c-6}{2}m \La_1- \frac{2b+c-6}{2}m \La_2=\frac{2b+c}{2}m R+ \sum_{P_i \in \La_1} 3mP_i + \sum_{Q_i \in \La_2} 3mQ_i .$$
If the dimension in \eqref{MYS1} is zero, we get a contradiction and we are done. If it is different from zero, 
by Lemma \ref {L:20210429-201}, by observing that  $(2b+c+1)m -1 -{\frac{2b+c-6}{2} m}= \frac{2b+c+8}{2}m -1$, and
$$ -2b+5c-8 =-2b+6c-c-8 \geq -2b+6c-(2b-4)-8 =4(c-b)+2c-4 \geq 2c
,$$
we have that $\La_1$ is a fixed component of multiplicity at least 
\begin{equation}\label{MM3} 
\left\lceil \frac{3 c m  -\frac{2b+c+8}{2}m +1}{c-1}\right\rceil=
\left\lceil \frac{(-2b+5c-8)m +2}{2(c-1)}\right\rceil \ge {m} ,
\end{equation}
for the  curves  of  the linear system $\big[I_{\frac{2b+c}{2}m R+ \sum_{P_i \in \La_1} 3mP_i + \sum_{Q_i \in \La_2} 3mQ_i }\big]_{\frac{2b+c+8}{2}m  -1}$.
We now remove $m \La_1$ and we get
$$
\begin{array}{llllll}
&\dim [I_{\frac{2b+c}{2}m R+ \sum_{P_i \in \La_1} 3mP_i + \sum_{Q_i \in \La_2} 3mQ_i }]_{\frac{2b+c+8}{2}m  -1}\\
=&
\dim [I_{\frac{2b+c}{2}m R+ \sum_{P_i \in \La_1} 2mP_i + \sum_{Q_i \in \La_2} 3mQ_i }]_{\frac{2b+c+8}{2}m  -1-m}.
\end{array}
$$
Finally, if the dimension above is zero,  we get a contradiction and we are done. If it is different from zero, then,
by Lemma \ref {L:20210429-201}, by recalling that $2b \geq c+ 4$, we have that $\La_2$ is a fixed component of multiplicity at least 
\begin{equation}\label{MM4} 
\left\lceil \frac{3m  b -\frac{2b+c+6}{2}m +1}{b-1}\right\rceil=
\left\lceil \frac{(4b-c-6)m +2}{2(b-1)}\right\rceil \geq
\left\lceil \frac{(2b+c+4-c-6)m +2}{2(b-1)}\right\rceil\ge {m} ,
\end{equation}
for the  curves  of  the linear system $ \big[I_{\frac{2b+c}{2}m R+ \sum_{P_i \in \La_1} 2mP_i + \sum_{Q_i \in \La_2} 3mQ_i }\big]_{\frac{2b+c+6}{2}m  -1}$.

Hence
$$
\begin{array}{llllll}
&\dim \big[I_{\frac{2b+c}{2}m R+ \sum_{P_i \in \La_1} 2mP_i + \sum_{Q_i \in \La_2} 3mQ_i }\big]_{\frac{2b+c+6}{2}m  -1}\\
=&
\dim \big[I_{\frac{2b+c}{2}m R+ \sum_{P_i \in \La_1} 2mP_i + \sum_{Q_i \in \La_2} 2mQ_i }\big]_{\frac{2b+c+4}{2}m  -1}.
\end{array}
$$
If this dimension is different from zero, then we go on and we apply   Lemma \ref {L:20210429-201} to the lines $\Ma_i$, $\Na_i$, and  $\Ta_i$. Since
\begin {equation}\label{MM5}
\left\lceil \frac{\frac{2b+c}{2}m  +2m  +2m-\frac{2b+c+4}{2}m +1}{2}\right\rceil=
\left\lceil \frac{2m +1}{2}\right\rceil>1, \ 
\hbox{and} \ 
\left\lceil \frac{\frac{2b+c}{2}m  +2m-\-\frac{2b+c+4}{2}m +1}{1}\right\rceil=1,
\end{equation}
the lines $\Ma_i$,   $\Na_i$, and $\Ta_i$ are fixed components for the  curves  of  the linear system 
$$\big[I_{\frac{2b+c}{2}m R+ \sum_{P_i \in \La_1} 2mP_i + \sum_{Q_i \in \La_2} 2mQ_i }\big]_{\frac{2b+c+4}{2}m  -1}.$$
Hence, from the computations in
(\ref {MM1}), (\ref {MM2}), (\ref {MM3}), 
(\ref {MM4}), and (\ref {MM5}),
we get that  the following curve 
\begin{equation}\label{compfissa-per-c-pari}
\frac{2b+c-4}{2}m \La_1  +\frac{2b+c-4}{2}m \La_2  +\Ma_1+\cdots+\Ma_{\frac{c}{2}}+\Na_1+\cdots+\Na_{\frac{c}{2}}+\Ta_{\frac{c+2}{2}}+  \cdots+\Ta_{b}
\end{equation}
is a fixed component for  the curves defined by  the linear system $\big[I_{\frac{2b+c}{2}m \X}\big]_{\frac{6b+3c-4}{2}m -1 }$.

Now set 
\begin {equation} \label {barm-per-c-pari}  \bar m =\min \{ m \ | \ [I_{\frac{2b+c}{2}m \X}]_{\frac{6b+3c-4}{2}m -1 }\neq\{0\} \}.
\end {equation}

First observe that $\bar m \neq 1$. In fact for 
$m =1$, the curve $\Fa'$ of degree $\frac {6b +3c-8}{2}$
$$\Fa'=\frac{2b+c-4}{2} \La_1  +\frac{2b+c-4}{2}\La_2   +\Ma_1+\cdots+\Ma_{\frac{c}{2}}+\Na_1+\cdots+\Na_{\frac{c}{2}}+\Ta_{\frac{c+2}{2}}+  \cdots+\Ta_{b}
$$
should be a fixed component for the linear system
$[I_{\frac{2b+c}{2} \X}]_{\frac{6b+3c-4}{2} -1 }$, so
$$\dim [I_{\frac{2b+c}{2} \X}]_{\frac{6b+3c-4}{2} -1 }=
\dim [I_{\frac{2b+c}{2} \X - \Fa'}]_{\frac{6b+3c-4}{2} -1- \frac {6b+3c -8}{2}}=
\dim [I_{P_1+\cdots+P_{c}+Q_1+\cdots+Q_{b}}]_1=0,$$
a contradiction.

So $\bar m>1$. By (\ref{compfissa-per-c-pari}), since $\frac{2b+c-4}{2}\bar m \geq \frac{2b+c-2}{2}$, we get that
$\Fa$ is a fixed component for the linear system
$[I_{\frac{2b+c}{2} \bar m\X}]_{\frac{6b+3c-4}{2}\bar m -1 }$, hence, by recalling that 
$\deg \Fa=\frac{6b+3c-4}{2}$ and $\Fa$  contains each point of   $\X$ with  multiplicity $\frac{2b+c}{2}$, we get
$$\dim [I_{\frac{2b+c}{2} \bar m\X}]_{\frac{6b+3c-4}{2}\bar m -1 }=
\dim [I_{\frac{2b+c}{2} \bar m\X- \Fa}]_{\frac{6b+3c-4}{2}\bar m -1- \frac {6b+3c -4}{2}}=
\dim [I_{\frac{2b+c}{2} (\bar m-1)\X}]_{\frac{6b+3c-4}{2}(\bar m -1)-1},
$$
which is zero by  (\ref {barm-per-c-pari}), a contradiction.
\end{proof}

\begin{thm}\label{MMcdispari} Let $\X$ be a { standard} $\k$-configuration of type  $(1,b,c)$. If $c$ is odd,  and $b+1 < c \le 2b-3$, then 
$$
\widehat\alpha(I_\X)=\frac{6b+3c-7}{2b+c-1}.
$$
\end{thm}

\begin{proof}
Let
$$
\Fa=\frac{2b+c-3}{2}\La_1 +\frac{2b+c-3}{2}\La_2 +\Ma_1+\cdots+\Ma_{\frac{c+1}{2}}+\Na_1+\cdots+\Na_{\frac{c-1}{2}}+\Ta_{\frac{c+3}{2}}+  \cdots+\Ta_{b}.
$$
$\Fa$ is the union of   $\frac{6b+3c-7}{2}$ lines, and $\Fa$  contains each point of   $\X$ with  multiplicity  exactly $\frac{2b+c-1}{2}$.
Hence, for $m >0$,
$$
m \Fa\in \big[I_{\frac{2b+c-1}{2}m \X}\big]_{\frac{6b+3c-7}{2}m }.
$$
It follows that $\widehat\alpha(I_\X)\leq  \frac{6b+3c-7}{2b+c-1}.$

Now, by recalling that $c-1>b$, we  can consider the standard $\k$-configuration $\X'$ of type $(1,b,c-1)$,
which is  contained in  the standard $\k$-configuration  $\X$. Hence 
$\widehat\alpha(I_\X) \geq \widehat\alpha(I_{\X'}) $.
Since  $c-1 \leq 2b-4 $ and $c-1$ is even, by Theorem  \ref{MMcpari} we have that 
$\widehat\alpha(I_{\X'})=\frac{6b+3(c-1)-4}{2b+(c-1)}= \frac{6b+3c-7}{2b+c-1},$ and  the conclusion  follows.
\end{proof}

\begin{cor}\label{ME20210909}
Let $\X$ and $\Y$ be   standard $\k$-configurations  of type $(1,b,c)$ and $(1,b,c+1)$,
respectively. If  $c$ is  even, and $c \leq 2b-4$, then
  $ \widehat\alpha(I_\X)= \widehat\alpha(I_\Y)$.
 \end{cor}
 \begin{proof} By Theorem \ref{MMcpari} we have that $\widehat\alpha(I_\X)=\frac{6b+3c-4}{2b+c}.$ Now by applying Theorem \ref{MMcdispari} to $\Y$ we get
 $\widehat\alpha(I_\Y) = \frac{6b+3(c+1)-7}{2b+(c+1)-1}=\frac{6b+3c-4}{2b+c}=\widehat\alpha(I_\X)$.
\end{proof}

From Theorems \ref{MMcpari} and \ref{MMcdispari},  we can compute the Waldschmidt constants of any standard $\k$-configurations of type  $(1,b,c)$, 
when $c\le 2b-3$, except for the configuration $\X$ of type $(1,b,b+1)$ with $b$ even. In the following theorem we will compute the Waldschmidt constant of this type of configuration, and we will find that $\widehat\alpha(I_\X) =\frac{9b-4}{3b}$.

 Alternatively  we could have considered the subscheme $\Y =\X- P_{b+1}$, and  computed the Waldschmidt constant of $\Y$, and found that $\widehat\alpha(I_\Y)=\frac{9b-4}{3b}$. With this method the conclusion would be followed from a theorem analogous to Theorem  \ref{MMcdispari}.

\begin{thm}\label{T:20210606-402}Let $\X$ be a { standard} $\k$-configuration  of type $(1,b,b+1)$. If $b \geq 4$ is an even  integer, then 
$$
\widehat\alpha(I_\X)=\frac{9b-4}{3b}.
$$
\end{thm} 
 
\begin{proof} The proof proceeds as in Theorem \ref{MMcpari}. See [Appendix \ref{appendice}, Proof of Theorem  \ref{T:20210606-402}] for more details.
\end{proof}

Now  we study  the standard $\k$-configurations  from  type $(1,b,2b-2)$ to  $(1,b,2b+1)$. In this range the Waldschmidt constant is strictly increasing.
A useful tool for the proofs is Lemma \ref{T:20210527-109}.
Also even if  the method is always the same,  we prefer 
to give some  details since the  proof is more tricky  than the previous cases.

\begin{thm}\label{M2b-2}
Let $\X$ be a standard $\k$-configuration  of type $(1,b,2b-2)$. Then
$$
\widehat \alpha(I_\X)=\frac{6b^2-14b+6}{2b^2-4b+1}.
$$
\end{thm} 

\begin{proof}
Note that from the  definition of  a standard $\k$-configuration, we have $b>2$.
Let
$$
\begin{array}{llllllllllllll}
\Ca_i & \text{be the irreducible curve of degree $(b-1)$ through $P_2,P_4,\dots, P_{2b-2},Q_1,\dots,\widehat Q_i,\dots, Q_b, (b-2)R$} \\
& \text{for $1\le i\le b-1$ (see Lemma~\ref{T:20210527-109}),} 
\end{array} 
$$
and let
$$
\Fa=(2b^2-5b+2)\La_1+(2b^2-6b+4)\La_2+(b-1)\Ma_1+\cdots+(b-1)\Ma_{b-1}+(b-2)\Ta_b+\Ca_1+\cdots+\Ca_{b-1}.
$$
So $\Fa$ is a curve of degree $6b^2-14b+6$ with  multiplicity $2b^2-4b+1$ at each point of $\X$. Hence for $m>0$
$$
m\Fa\in [I_{(2b^2-4b+1)m\X}]_{(6b^2-14b+6)m}.
$$
We  prove that for $m >0$,
$$
 [I_{(2b^2-4b+1)m\X}]_{(6b^2-14b+6)m-1}= \{0\}.
$$
\medskip
Assume that for some $m$, 
$
 [I_{(2b^2-4b+1)m\X}]_{(6b^2-14b+6)m-1}\neq \{0\}.
$
Thus by Lemma  \ref {L:20210429-201},  $\La_1$ is a fixed component of multiplicity at least 
\begin{equation}\label{M2b-2-1} 
\left\lceil \frac{(2b^2-4b+1)(2b-2)m-(6b^2-14b+6)m+1}{2b-3}\right\rceil \geq (2b^2-6b+3)m
\end{equation}
for the plane curves of the linear system
$ [I_{(2b^2-4b+1)m\X}]_{(6b^2-14b+6)m-1}$.
We remove  $ (2b^2-6b+3)m \La_1$, and we get that $\La_2$ is a fixed component of multiplicity at least 
\begin{equation}\label{M2b-2-2} 
\left\lceil \frac{(2b^2-4b+1)bm-(6b^2-14b+6-2b^2+6b-3)m+1}{b-1}\right\rceil \geq (2b^2-6b+3)m.
\end{equation}
Remove $(2b^2-6b+3)m\La_2$. Recalling that  now we are in degree
$(6b^2-14b+6)m-2(2b^2-6b+3)m-1=(2b^2-2b)m-1$, and  the points on $\La_1$ have multiplicity $(2b-2)m$,
we get that $\La_1$ is a fixed component of multiplicity at least 
\begin{equation}\label{M2b-2-3} 
\left\lceil \frac{(2b-2)(2b-2)m-(2b^2-2b)m+ 1}{2b-3}\right\rceil = (b-2)m+
\left\lceil \frac{(b-2)m+1}{2b-3}\right\rceil.
\end{equation}
Hence $\La_1$ is a fixed component of multiplicity at least $(2b^2-6b+3)m+ (b-2)m=(2b^2-5b+1)m$.

By removing $(b-2)m \La_1$  we get
$$\dim [I_{(2b^2-4b+1)m\X}]_{(6b^2-14b+6)m-1}=
\dim [I_{(2b^2-4b+1)m R+ \sum_{P_i \in \La_1} bmP_i + \sum_{Q_i \in \La_2} (2b-2)mQ_i }]_{(2b^2-3b+2)m-1}.
$$
If the above dimension is different from zero, then each $\Ma_i$ is a fixed component of multiplicity at least 
\begin{equation}\label{M2b-2-4} 
\left\lceil \frac{(2b^2-4b+1+b+2b-2)m-(2b^2-3b+2)m+ 1}{2}\right\rceil =
 (b-2)m + \left\lceil \frac{m+1}{2}\right\rceil.
\end{equation}
By removing the $b-1$ multiple lines $(b-2)m\Ma_i$, the residual scheme is
$$\Y=
(b^2-b-1)m R
+ \sum_{P_i \in \La_1, \text{ with } i  \text{ odd }\ } 2mP_i 
+\sum_{P_i \in \La_1,  \text{ with } i    \text{ even }} bmP_i 
+ \sum_{i=1} ^{b-1}bmQ_i +(2b-2)mQ_b,
$$
and we are left in degree $(2b^2-3b+2)m-1-(b-1)(b-2)m=b^2m-1.$

Hence
$$\dim [I_{(2b^2-4b+1)m\X}]_{(6b^2-14b+6)m-1}=
\dim [I_{\Y}]_{b^2m-1}.
$$
If this dimension is still different from zero, then  $\Ta_b$ is a fixed component of multiplicity at least 
\begin{equation}\label{M2b-2-5} 
(b^2-b-1+2b-2)m-b^2m+1 = (b-3)m+1.
\end{equation}

By removing $(b-3)m\Ta_b$ we get
$$\dim [I_{(2b^2-4b+1)m\X}]_{(6b^2-14b+6)m-1}=
\dim [I_{\Y-(b-3)m\Ta_b}]_{b^2m-1-(b-3)m}
=\dim [I_{\Y'}]_{(b^2-b+3)m-1},
$$
where
$$\Y'=
(b^2-2b+2)m R
+ \sum_{P_i \in \La_1, \text{ with } i  \text{ odd }\ } 2mP_i 
+\sum_{P_i \in \La_1,  \text{ with } i    \text{ even }} bmP_i 
+ \sum_{i=1} ^{b-1}bmQ_i +(b+1)mQ_b.
$$
If $\Ha$ is a curve of the linear system $[I_{\Y'}]_{(b^2-b+3)m-1}$,
the multiplicity of intersection between each $\Ca_i$ and $\Ha$ is at least 
$$|\Ca_i \cdot \Ha | \geq 
(b-2)(b^2-2b+2)m  
+(b-1) bm
+ (b-2)bm+(b+1)m=(b^3-2b^2+4b-3)m,
$$
and this number is bigger than the product of the degree of $\Ca_i$ and $\Ha$, which is
$(b-1)((b^2-b+3)m-1)=(b^3-2b^2+4b-3)m -b+1. $
Hence, by B\'ezout's Theorem, each curve $\Ca_i$ is a fixed component for the curves of  $[I_{\Y-(b-3)m\Ta_b}]_{b^2m-1-(b-3)m}.$

Now let 
\begin {equation} \label {barm-per-2b-2} 
 \bar m =\min \{ m |\ [I_{(2b^2-4b+1)m\X}]_{(6b^2-14b+6)m-1}\neq\{0\} \}.
\end {equation}
We have $\bar m>1$, in fact for $m=1$, by \eqref{M2b-2-1}, 
\eqref{M2b-2-2}, \eqref{M2b-2-3}, \eqref{M2b-2-4}, \eqref{M2b-2-5}, and
 using also the ceiling parts, we get that $\Fa$ is a curve of the linear system should be a fixed component for the linear system
$  [I_{(2b^2-4b+1)\X}]_{6b^2-14b+5}$, but $\deg \Fa = 6b^2-14b+6$,
a contradiction.

Hence $\bar m>1$. 

By  the computation above
$\Fa$ is a fixed component for the linear system
$ [I_{(2b^2-4b+1)m\X}]_{(6b^2-14b+6)m-1},$
hence we have
$$
\begin{array}{llllll}
\dim   [I_{(2b^2-4b+1)\bar m\X}]_{(6b^2-14b+6)\bar m-1}&=&
 \dim  [I_{(2b^2-4b+1)\bar m\X-\Fa}]_{(6b^2-14b+6)\bar m-1-(6b^2-14b+6)}\\
 &=&
 \dim  [I_{(2b^2-4b+1)(\bar m-1)\X}]_{(6b^2-14b+6)(\bar m-1)-1} ,
 \end{array}
$$
which is zero by (\ref {barm-per-2b-2} ), a contradiction.
\end{proof}

\begin{thm}\label{T:20210717}
Let $\X$ be a standard $\k$-configuration  of type $(1,b,2b-1)$. 
Then 
$$
\widehat \alpha(I_\X)=\frac{6b^2-8b+1}{2b^2-2b}.
$$
\end{thm} 

\begin{proof} See [Appendix \ref{appendice}, Proof of Theorem  \ref{T:20210717}].
\end{proof}

\begin{thm} \label{T:20210428-203}
Let $\X$ be a  standard $\k$-configuration in  of type $(1,b,2b)$. Then 
$$
\widehat\alpha (I_\X)=\frac{6b-5}{2b-1}.
$$
\end{thm}

\begin{proof} See [Appendix \ref{appendice},  Proof of Theorem  \ref{T:20210428-203}].
\end{proof} 
\medskip

\begin{thm} \label{T:20210618-203}
Let $\X$ be a {standard} $\k$-configuration in  of type $(1,b,2b+1)$. Then $$
\widehat\alpha (I_\X)=\frac{6b^2-2b-3}{2b^2-1}.
$$
\end{thm}

\begin{proof} See [Appendix \ref{appendice}, Proof of Theorem  \ref{T:20210618-203}].
\end{proof} 

 Now we will prove that the Waldschmidt constant of  a standard $\k$-configuration of type $(1,b,c)$ only depends on $b$ when $c \geq 2b+2$.  In order to do that, we need the following lemma.

\begin{lem} \label{L:20210616-418} 
Let $L_1$, $L_2$ be two distinct  lines, and let $b$, $c$ be  positive integers,  with $c\geq b+2$. Let  $P_1,\dots,P_{c} \in L_1$,  $Q_1,\dots,Q_b\in L_2$, and $R$, be distinct points such that 
 $R \not\in L_1 \cup L_2$, 
and the point  $L _1\cap L_2 \not\in \{P_1,\dots,P_c , Q_1,\dots,Q_b\}$.
Moreover $R$, $P_i$, $Q_j$ do not lie on a line, for any $i$ and $j$. Let $\Y_c$  be the scheme
(see Figure \ref{FIG:20210704-005})

$$\Y_c = P_1+\cdots+P_{c} +Q_1+\cdots +Q_b +R.$$
Then
$$
\widehat\alpha(\Y_c)=\frac{3b-1}{b}.
$$
\end{lem}

\begin{proof}  

\begin{figure}[ht] 
\centering
\vskip -3.3cm
\psset{xunit=.8cm,yunit=.8cm,algebraic=true,dimen=middle,dotstyle=o,dotsize=7pt 0,linewidth=1pt,arrowsize=3pt 2,arrowinset=0.25}
\begin{pspicture*}(-7.14,-9.2)(12.4,8.6)

\psline[linewidth=1pt](0.04,4.38)(7.9,4.38)
\psline[linewidth=1pt](0.34,2.32)(6.,3.5)

\rput[tl](-0.33,3.55){\tiny$R$}
\rput[tl](-0.48,4.53){\tiny$L_1$}
\rput[tl](-0.16,2.45){\tiny$L_2$}
\rput[tl](2.79,4.59){\Huge$\cdots$}
\rput[tl](2.55,3.02){\Huge$\cdot$}
\rput[tl](3.05,3.13){\Huge$\cdot$}
\rput[tl](3.55,3.24){\Huge$\cdot$}

\begin{scriptsize}
\psdots[dotstyle=*](0.2,3.42)
\psdots[dotstyle=*](1,4.4)
\psdots[dotstyle=*](2.0,4.38)
\psdots[dotstyle=*](5.0,4.38)
\psdots[dotstyle=*](6.0,4.38)
\psdots[dotstyle=*](7.0,4.38)
\psdots[dotstyle=*](1.2713797048392632,2.514939938222171)
\psdots[dotstyle=*](2.1432845212218283,2.6974316439766617)
\psdots[dotstyle=*](4.268294245509667,3.142201121153186)
\psdots[dotstyle=*](5.087365290012585,3.313634595584029)
\end{scriptsize}
\end{pspicture*}\vskip -9.4cm
\caption{The scheme $\Y_c$ }\label{FIG:20210704-005}
\end{figure}
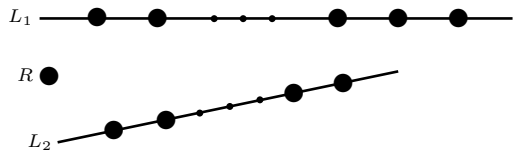

If $b=1$, $\Y_c$ is  a  $\k$-configuration of type $(2,c)$, hence $\widehat\alpha(\Y_c)=2$ follows from Corollary \ref {P:20210529-206}. 
The proof for  $b=2$ is analogous to the proof for $b>2$, and it is left to the reader,
so assume $b>2$.

First we prove the lemma for  $c=b+2$. For this case, we denote 
$\Y_{b+2}$ simply by $\Y$.
Let
$M_i$ be the line through $Q_i$ and $R$, ( $1\le i \le b$), 
and let
$$
\Fa=bL_1+(b-1)L_2+M_1+\cdots+M_b.
$$
 Note that $\deg \Fa=3b-1$ , and $\Fa$ has  multiplicity exactly $b$ at all points of $\Y$. Hence for $m >0$
$$
m \Fa \in [I_{bm \Y}]_{(3b-1)m}.
$$
Now we will show that for $m>0$, 
$$
[I_{bm \Y}]_{(3b-1)m -1} = \{0\},
$$
and the conclusion will  follow from Lemma \ref {L:20210722-1}.

Assume that for some $m >0$, 
$
[I_{bm \Y}]_{(3b-1)m -1} \neq \{0\}.
$

By Lemma \ref {L:20210429-201},  
$L_1$ is a fixed component of multiplicity at least 
$$
\left\lceil \frac{b(b+2)m-(3b-1)m+1}{b+1} \right\rceil
\ge (b-2)m.
$$
So we can remove $(b-2)mL_1$, and we get that
$$\dim[I_{bm \Y}]_{(3b-1)m -1} = \dim [I_{bm \Y-(b-2)mL_1}]_{(2b+1)m -1} . $$
If this dimension is different from zero, we get that 
$L_2$ is a fixed component of multiplicity at least 
$$
\left\lceil \frac{b^2m-(2b+1)m+1}{b-1} \right\rceil =  (b-2)m+ 
\left\lceil \frac{(b-3)m+1}{b-1} \right\rceil ,
$$
and then that $L_1$ is a fixed component of multiplicity at least 
$$
\left\lceil \frac{2(b+2)m-(b+3)m+1}{b+1} \right\rceil =  m +\left\lceil \frac{1}{b+1} \right\rceil
.$$
Hence
$$
\begin{array}{lllllll}
\dim [I_{bm\Y}]_{(3b-1)m -1} &=&
\dim [I_{bm \Y-(b-1)m\La_1-(b-2)m\La_2}]_{(b+2)m -1}\\
&=&\dim [I_{\sum_{i=1}^{b+2} mP_i+\sum_{i=1}^b 2mQ_i+bm R}]_{(b+2)m -1}.
\end{array}
$$
Now, by Bezout's Theorem, each $M_i$ is a fixed component ( $1\le i \le b$) for $[I_{bm \Y}]_{(3b-1)m -1}$. 
Thus, from the equalities above, we have that $\Fa$ is a curve of degree $ 3b-1$ of the linear system $[I_{b \Y}]_{3b-2}$, a contradiction.

Hence $\bar m>1$, where
\begin {equation} \label {barm-per-lemma}  \bar m =\min \{ m |[I_{mb\Y}]_{m (3b-1)-1}\neq\{0\} \}.
\end {equation}
Now from the computation above,
$\Fa$ is a fixed component for the linear system $[I_{\bar mb\Y}]_{\bar m (3b-1)-1}$,
hence 
$$\dim [I_{\bar mb\Y}]_{\bar m (3b-1)-1}=
\dim [I_{\bar mb\Y-\Fa}]_{\bar m (3b-1)-1-(3b-1)}
=\dim [I_{(\bar m-1)b\Y}]_{(\bar m-1) (3b-1)-1},$$
which is zero by  (\ref {barm-per-lemma}), a contradiction.

Now consider the case $c>b+2$.
Since also in this case $m \Fa \in [I_{bm \Y}]_{(3b-1)m}$, then $\widehat\alpha(\Y_c) \leq \frac{3b-1}{b}$.
Moreover, since $\Y _{b+2}\subset \Y_{c}$, then $\widehat\alpha(\Y_{b+2})\leq \widehat\alpha(\Y_c),$ and the conclusion follows.
\end{proof} 

\begin{thm}\label{T:20210617-419}
Let $\X$ be a { standard} $\k$-configuration  of type $(1,b,c)$  with 
$c \geq 2b+2$. Then 
$$
\widehat\alpha(I_\X)=\frac{3b-1}{b}.
$$
\end{thm} 

\begin{proof} Let us consider the following curve $\Fa$ of degree $(3b-1)$ with  multiplicities at least $b$  at the points in $\X$
$$
\Fa=b\La_1+(b-1)\La_2+\Ma_1+\cdots+\Ma_b.
$$
Then, for $m>0$, we have
$
m\Fa\in [I_{mb\X}]_{(3b-1)m}.
$
This implies,
$$
\widehat \alpha(I_\X)\le \frac{3b-1}{b}.
$$
To conclude the proof set
$
\Y=\X-\{P_1,P_3,\dots,P_{2b-1}\}.
$
Then, by Lemma~\ref{L:20210616-418} and since $\Y\subseteq \X$, we get
$$\frac{3b-1}{b}= \widehat \alpha(I_\Y) \le \widehat \alpha(I_\X)\le \frac{3b-1}{b}.$$
This completes the proof.
\end{proof}

\section{Waldschmidt constants of  standard $\k$-configurations of type $(a,b,c)$, with $a\geq 2$.}\label{Sabc}

In this section we  study the Waldschmidt constant of  a standard $\k$-configuration of type $(a,b,c)$, with $a \geq 2$. We  prove that, except for the type  $(2,3,4)$, and  for the type  $(2,3,5)$ (see Theorem \ref{M20210728-234} and Remark~\ref{rem235}), then   the Waldschmidt constant is $3$.
For this section we fix the following notation (see Figure \ref{figuragenerale2}).

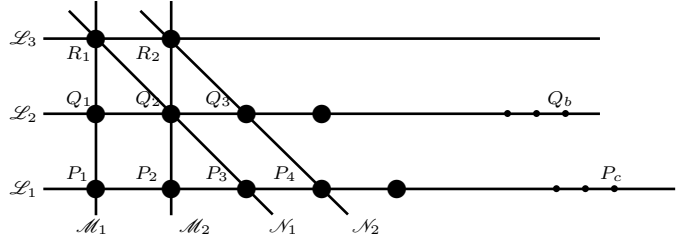
\begin{figure}[ht] 
\centering
\vskip -5.5cm
\psset{xunit=1cm,yunit=1cm,algebraic=true,dimen=middle,dotstyle=o,dotsize=7pt 0,linewidth=.5pt,arrowsize=3pt 2,arrowinset=0.25}
\begin{pspicture*}(-10,-8.8)(5,8.8)

\psline[linewidth=1pt](-6.7,1)(1.7,1)

\psline[linewidth=1pt](-6.7,2)(0.7,2)

\psline[linewidth=1pt](-6.7,3)(0.7,3)

\psline[linewidth=1pt](-6,3.5)(-6,0.65)

\psline[linewidth=1pt](-5,3.5)(-5,0.65)

\psline[linewidth=1pt](-6.36,3.37)(-3.65,0.66)

\psline[linewidth=1pt](-5.41,3.37)(-2.65,0.66)

\rput[tl](-0.65,2.17){\Huge $\cdots$}
\rput[tl](0.0,1.17){\Huge$\cdots$}

\rput[tl](-7.15,1.1){\tiny$\La_1$}
\rput[tl](-7.15,2.1){\tiny$\La_2$}
\rput[tl](-7.15,3.1){\tiny$\La_3$}

\rput[tl](-6.40,2.9){\tiny$R_1$}
\rput[tl](-5.48,2.9){\tiny$R_2$}
\rput[tl](-6.40,2.3){\tiny$Q_1$}
\rput[tl](-5.48,2.3){\tiny$Q_2$}
\rput[tl](-4.55,2.3){\tiny$Q_3$}
\rput[tl](-0.00,2.3){\tiny$Q_b$}

\rput[tl](-6.40,1.3){\tiny$P_1$}
\rput[tl](-5.48,1.3){\tiny$P_2$}
\rput[tl](-4.55,1.3){\tiny$P_3$}
\rput[tl](-3.65,1.3){\tiny$P_4$}
\rput[tl](0.7,1.3){\tiny$P_c$}

\rput[tl](-6.27,0.6){\tiny$\Ma_1$}
\rput[tl](-4.9,0.6){\tiny$\Ma_2$}
\rput[tl](-3.7,0.6){\tiny$\Na_1$}
\rput[tl](-2.6,0.6){\tiny$\Na_2$}

\begin{scriptsize}
\psdots[dotstyle=*](-6,3)
\psdots[dotstyle=*](-5,3)
\psdots[dotstyle=*](-6,2)
\psdots[dotstyle=*](-5,2)
\psdots[dotstyle=*](-4,2)
\psdots[dotstyle=*](-3,2)
\psdots[dotstyle=*](-6,1)
\psdots[dotstyle=*](-5,1)
\psdots[dotstyle=*](-4,1)
\psdots[dotstyle=*](-3,1)
\psdots[dotstyle=*](-2,1)
\end{scriptsize}
\end{pspicture*}
\vskip -9.4cm
\caption{A standard $\k$-configuration of type $(2,b,c)$}\label{figuragenerale2}
\end{figure}

Let  $P_i=[1: i-1: 0]$, for $1\le i\le c$, let $Q_i=[1:i-1 :1]$, for $1\le i\le b$,  let $R_1=[1:0:2]$  and $R_2=[1:1:2]$  be the points of $\X$, and let 
$$
\begin{array}{llllllllllllllllll}
\La_1 & \text{be the line through  $P_1,P_2,\dots,P_{c}$;} \\
\La_2 & \text{be the line through  $Q_1,Q_2,\dots,Q_{b}$;} \\
\La_3 & \text{be the line through  $R_1,R_2$;} \\
\Ma_1 & \text{be the line through  $P_1,Q_1,R_1$;} \\
\Ma_2 & \text{be the line through  $P_2,Q_2,R_2$;} \\
\Na_1 & \text{be the line through  $P_3,Q_2,R_1$;} \\
\Na_2 & \text{be the line through  $P_4,Q_3,R_2$.} \\
\end{array}
$$
First we compute the Waldschmidt constant of a $\k$-configuration  of type $(2,b,c)\neq (2,3,5)$.

\begin{thm} \label{M20210728-234}
Let $\X$ be a {standard} $\k$-configuration  of type $(2,3,4)$. Then the Waldschmidt constant of $\X$ is
$$
\widehat\alpha (I_\X)=\frac{17}{6}.
$$
\end{thm}
\begin{proof}
Let 
$$
\begin{array}{lll}
\Ca     \text{ be the conic through  $P_2,P_3,Q_1,Q_3,R_1,R_2$},
\end{array}
$$
and let $\Fa$ be the following curve of degree $17$, which contains each point of   $\X$ with  multiplicity $6$
$$\Fa = 3 \La_1+2\La_2+3\Ma_1+2\Ma_2+2\Na_1+3\Na_2+ \Ca.
$$
Hence, for $m >0$,
$$
m \Fa\in [I_{6m \X}]_{17m }.
$$
 The conclusion will follows from Lemma \ref {L:20210722-1},
if we prove that for each $m >0$,
$$
\dim  [I_{6m \X}]_{17m-1 }=0.
$$
As usual, assume that for some $m$, $[I_{6m \X}]_{17m-1 }\neq \{0\}$.
By Lemma \ref {L:20210429-201},  $\La_1$ is a fixed component of multiplicity at least 
$\left\lceil \frac{24m-17m +1}{3}\right\rceil= \left\lceil \frac{7m +1}{3}\right\rceil \geq 2m $
for the plane curves of the linear system
$[I_{6m \X}]_{17m-1 }$.
By removing $2m\La_1$ and assuming that the residual linear system is not empty, by Lemma~\ref {L:20210429-201}, we get that $ \La_2$ is a fixed component of multiplicity at least $\left\lceil \frac{3m +1}{2}\right\rceil $, and 
$ \Ma_1$, $\Ma_2$, $ \Na_1$, $\Na_2$ are fixed component of multiplicity at least $\left\lceil \frac{m +1}{2}\right\rceil $.
Let 
\begin {equation} \label {barm7}  \bar m =\min \{ m | [I_{6m \X}]_{17m -1 }\neq\{0\} \}.
\end {equation}
Now we claim that  for $m = 1,2,3$, $[I_{6m \X}]_{17m -1 }=\{0\}$.  This claim can be proved directly, with the usual method.   It follows that $\bar m \geq 4$.

From the computation above,  and recalling that $ \Ma_1$, $\Ma_2$, $ \Na_1$, $\Na_2$ are fixed component of multiplicity at least $\left\lceil \frac{\bar m +1}{2}\right\rceil \geq 3$, then
$\Fa$ is a fixed component for the linear system $[I_{6\bar m \X}]_{17\bar m -1 }$,
hence 
$$\dim [I_{6\bar m \X}]_{17\bar m -1 }=
\dim [I_{6\bar m \X-\Fa}]_{17\bar m -1-17 }
=\dim [I_{6(\bar m -1)\X}]_{17(\bar m -1)-1 },$$
which is zero by (\ref {barm7}), a contradiction.
\end{proof}

We need the following lemma  to find out the Waldschmidt constant of a {standard} $\k$-configuration of type~$(2,3,6)$.

\begin{lem} \label{Giuseppe} 

Let $L_1$, $L_2$ be two distinct  lines, and let $P_1,\dots,P_{6} \in L_1$,  and $Q_1,Q_2,Q_3\in L_2$  be distinct points such that $ L_1 \cap L_2 \not\in \Y$, where 
$$\Y = P_1+\cdots+P_{6} +Q_1+\cdots +Q_3 .$$
 Let $m$ be a positive integer, then the curve $2m L_1+m L_2$ is a fixed component for the linear system
$
[I_{3m \Y}]_{9 m -1 }.$

\end{lem} 
\begin {proof} 
Set
$$M=
\{ m' \ |\  2 m' L_1 +  m' L_2   \hbox { is a fixed component for the linear system} \ \ 
[I_{3m\Y}]_{9 m -1 }\}.$$
Since by Lemma \ref {L:20210429-201}, $L_1$ and $L_2$ are  fixed components of multiplicity at least $\left\lceil \frac{18m-9m+1}{5}\right\rceil\geq 2 $, and 
$\left\lceil \frac{9m-9m+1}{2}\right\rceil=1 $, respectively, then $2  L_1 +  L_2$ is a fixed component for $[I_{3m\Y}]_{9 m -1 }$, and so
$1 \in M$.
Let $$ \bar m=\max M .$$
If $\bar m\geq m$ we are done, so assume  that  $\bar m<m$. By the definition of $\bar m$, we have that 
$2 \bar mL_1 +  \bar m L_2 $ is a fixed component for the linear system $[I_{3m \Y}]_{9 m -1 } $. Hence
$$
[I_{3m \Y}]_{9 m -1 } =H\cdot [I_{3m \Y-2 \bar m\La_1 -  \bar m \La_2}]_{9 m -1 -3 \bar m} = H\cdot [I_{\sum_{i=1}^6 P_i (3m-2\bar m) +\sum_{i=1}^3 (3m-\bar m) Q_i}]_{9 m -1 -3 \bar m} ,
$$
where $H$ is a form representing the curve $2 \bar mL_1 +  \bar m L_2 $ .
Now, by Lemma \ref {L:20210429-201}, we get that, for the curve of the linear system 
$[I_{3m \Y-2 \bar mL_1 -  \bar m L_2}]_{9 m -1 -3 \bar m}$, 
 $L_1 $ is a fixed component of multiplicity at least 
$$\left\lceil \frac{6(3m-2\bar m)-(9 m -1 -3 \bar m)}{5}\right\rceil
= \left\lceil \frac{9m -9 \bar m +1 }{5}\right\rceil\geq2 ,$$ and $L_2 $ is a fixed component 
of multiplicity at least 
$$\left\lceil \frac{3(3m-\bar m)-(9 m -1 -3 \bar m)}{2}\right\rceil=1 .$$
Hence  $2  L_1 +  L_2$ is a fixed component for $[I_{3m \Y-2 \bar mL_1 -  \bar m L_2}]_{9 m -1 -3 \bar m}$ and so, 
$2 (\bar m+1)L_1 +  (\bar m +1)L_2 $ is a fixed component for the linear system $[I_{3m \Y}]_{9 m -1 } $, a contradiction.
\end{proof}

\begin{thm} \label{M20210730-236}
Let $\X$ be a  standard $\k$-configuration  of type $(2,3,6)$. Then $$
\widehat\alpha (I_\X)=3.
$$
\end{thm}

\begin{proof}
Let $\Fa$ be the following curve of degree $9$, which contains each point of   $\X$ with  multiplicity $3$,
$$\Fa = 3 \La_1+3\La_2+3\La_3.
$$
Hence, for $m >0$,
$$
m \Fa\in [I_{3m \X}]_{9m }.
$$
 The conclusion will follows from Lemma \ref {L:20210722-1},
if we prove that for each $m >0$,
$$
\dim  [I_{3m \X}]_{9m-1 }=0.
$$
Assume that for some $m$, $[I_{3m \X}]_{9m-1 }\neq \{0\}$.
By Lemma \ref {Giuseppe},  $2m\La_1+m\La_2$ is a fixed component  for
$[I_{3m \X}]_{9m-1 }$, hence
$$\dim[I_{3m \X}]_{9m-1 }=\dim[I_{3m \X-2m\La_1-m\La_2}]_{9m-1-3m }
=\dim[I_{\sum _{i=1}^6m P_i +\sum _{i=1}^3 2m Q_i +\sum _{i=1}^2 3m R_i}]_{6m-1 }.
$$
Now if we  prove that this last dimension is zero,
we get a contradiction.
 
 {\it Claim}. 
 $$\dim [I_{\sum _{i=1}^6m P_i +\sum _{i=1}^3 2m Q_i +\sum _{i=1}^2 3m R_i}]_{6m-1 }
=0, \ \ \hbox {for each } \ m \geq 1.$$
We prove the claim by induction on $m$. It is easy to verify that it is true for $m=1$, so assume $m>1$. 
If this dimension is not zero, by Bezout's Theorem, $\La_1$, $\La_2$, $\La_3$ are  fixed components , hence
$$\dim [I_{\sum _{i=1}^6m P_i +\sum _{i=1}^3 2m Q_i +\sum _{i=1}^2 3m R_i}]_{6m-1 }=
\dim [I_{\sum _{i=1}^6 (m-1) P_i +\sum _{i=1}^3 (2m-1) Q_i +\sum _{i=1}^2 (3m-1) R_i}]_{6m-4 }.$$
If this dimension is still not zero, 
by Lemma \ref {L:20210429-201}, $\La_2$ and $\La_3$ are  fixed components of multiplicity at least 
$\left\lceil \frac{6m-3-(6m-4)}{2}\right\rceil =1,$
and 
$\left\lceil \frac{6m-2-(6m-4)}{1}\right\rceil= 2,$ respectively.
Hence
$$
\begin{array}{llll}
 &\dim [I_{\sum _{i=1}^6 (m-1) P_i +\sum _{i=1}^3 (2m-1) Q_i +\sum _{i=1}^2 (3m-1) R_i}]_{6m-4 } \\
=&\dim [I_{\sum _{i=1}^6 (m-1) P_i +\sum _{i=1}^3 2(m-1) Q_i +\sum _{i=1}^2 3(m-1) R_i}]_{6(m-1)-1 },
\end{array}$$
and this is zero by the inductive hypothesis.
\end{proof}

\begin{thm} \label{M20210730-245}
Let $\X$ be a  standard $\k$-configuration  of type $(2,4,5)$. Then the Waldschmidt constant of  $\X$ is
$$
\widehat\alpha (I_\X)=3.
$$
\end{thm}

\begin{proof}
Let $\Fa$ be the following curve of degree $6$, which contains each point of   $\X$ with  multiplicity $2$,
$$\Fa = 2 \La_1+2\La_2+2\La_3.
$$
Hence, for $m >0$,
$$
m \Fa\in [I_{2m \X}]_{6m }.
$$
Now, as usual, we have to
 prove that for each $m >0$,
$
\dim  [I_{2m \X}]_{6m-1 }=0.
$
It is true for $m=1$, assume $m>1$.
Assume that for some $m$, $[I_{2m \X}]_{6m-1 }\neq \{0\}$, and let
$$\bar m = \min 
\{ m | \dim  [I_{2\bar m \X}]_{6\bar m-1} \} \neq 0.
$$
By Lemma \ref {L:20210429-201},  
$\La_1$ is a fixed component of multiplicity at least 
$
\left\lceil \frac{10 \bar m-6\bar m+1}{4} \right\rceil
\ge \bar m+1.
$
Hence
$$
\dim  [I_{2 \bar m \X}]_{6\bar m-1 }=\dim  [I_{2\bar m \X-(\bar m+1)\La_1}]_{5\bar m-2 }.
$$
If this dimension is not zero, we get that  $\La_2$ is a fixed component of multiplicity at least 
$
\left\lceil \frac{8 \bar m-5\bar m+2}{3} \right\rceil
\ge \bar m+1
$.
Hence
$$
\dim  [I_{2\bar m \X}]_{6\bar m-1 }=\dim  [I_{2\bar m \X-(\bar m+1)\La_1-(\bar m+1)\La_2}]_{4\bar m-3 }.
$$
If this dimension is not zero, we get that  $\La_3$ is a fixed component of multiplicity at least 
$
\left\lceil \frac{4\bar m-4\bar m+3}{1} \right\rceil =3.
$
It follows that $\Fa$ is a fixed component. Hence, we get a contradiction since
$$
\dim  [I_{2\bar m \X}]_{6\bar m-1 }=\dim  [I_{2\bar m \X-\Fa}]_{6\bar m-1 -6}
=\dim  [I_{2(\bar m-1) \X}]_{6(\bar m-1) -1},
$$
which is zero by the definition of $\bar m$.
\end{proof}

\begin{thm} \label{M20210730-2bc}
Let $\X$ be a { standard} $\k$-configuration  of type $(2,b,c)$. 

\begin {itemize}
\item [(i)] If $b=3$ and $c \geq 6$, then $\widehat\alpha (I_\X)=3;$
\item [(ii)] if $b \geq 4$, then $\widehat\alpha (I_\X)=3.$
\end {itemize}
\end{thm}
\begin {proof} 

Let $\Fa=  \La_1+\La_2+\La_3$. Since  $m\Fa \in [I_{m \X}]_{3m },$ then in both cases, 
$\widehat\alpha (I_\X)\leq 3.$

Now let $\X$ be a  standard $\k$-configuration of type $(2,3,c)$, with $c \geq 6$. Then 
there exists a  standard $\k$-configuration $\X'$ of type $(2,3,6)$, with $\X' \subseteq \X$ .
Since, by Theorem \ref {M20210730-236}, the Waldschmidt constant of $\X'$ is $3$,  then 
$\widehat\alpha (I_\X)\geq 3$, and (i) is proved.

For (ii), since $b \geq 4$, then 
there exists a  standard $\k$-configuration $\X' $ of type $(2,4,5)$, with $\X' \subseteq \X$. Since,
by Theorem \ref {M20210730-245}, the Waldschmidt constant of $\X'$ is $3$, hence 
$\widehat\alpha (I_\X)\geq 3$, and (ii) is proved.
\end{proof}

\begin {rem} \label {rem235} From the previous results we know the Waldschmidt constant of any standard $\k$-configuration of type $(2,b,c)$, except for $\X$ of type  $(2,3,5)$.
 For the case $(2,3,5)$, we found by Macaulay 2 \cite{M2} a curve $\Fa$ of degree $71$ with multiplicity exactly $24$ at each point of $\X$. 
The components of  $\Fa$  are  lines, one irreducible conic and an irreducible rational septic.
This implies  $\widehat\alpha(I_{\X})\le \frac{71}{24}<3$. 
 Moreover, since a $\k$-configuration of type $(2,3,4)$ is a subset of $\X$, this  give  $\frac{17}{6}$ as a lower bound (see Theorem \ref{M20210728-234}). Hence   $\frac{17}{6}\leq \widehat\alpha(I_{\X})\le \frac{71}{24}$). 

\end {rem} 

Finally, we deal with the $\k$-configurations  of type $(a,b,c)$ when $a \geq 3$.

\begin{thm} \label{M20210730-abc}
Let $\X$ be a  standard $\k$-configuration  of type $(a,b,c)$, whith $a \geq 3$. Then the Waldschmidt constant of $\X$ is
$$
\widehat\alpha (I_\X)=3.
$$
\end{thm}

\begin{proof}
It follows immediately from Corollary~\ref{M20210802-1}.
\end{proof}

\begin{rem}\label{R:20230725} We  recall the Chudnovsky’s Conjecture:

{\it Let $\X$ be a finite set of distinct points in $\P^n$. Then,  for all $m>0$, 
$$
\frac{ \alpha(I^{(m)}_\X)}{m}\ge \frac{\alpha(I_\X)+n-1}{n}. 
$$}
This conjecture was proved  in $\P^2$ by Chudnovsky (see, for instance  \cite[Proposition 3.1]{HH}).
As an application, we wish to show that the Chudnovsky’s conjecture is verified by a standard $\k$-configurations in $\P^2$ of type $(a,b,c)$. 

Let $\X$ and $\Y$ be  standard $\k$-configurations in $\P^2$ of type $(a,b,c)$, and  $(b,c)$, respectively.  We know that 
$\alpha(I_{\X})= 3$ , and from the proof of Lemma \ref{L:20210529-205}, recalling that $b >1$, we get that $\alpha(I^{(m)}_{\Y})= 2m$.
Moreover, since  the scheme $m\X \supset m\Y$, then $\alpha(I^{(m)}_{\X})  \geq \alpha(I^{(m)}_{\Y})$. 
It follows that,  for all $m>0$,
$$
\frac{ \alpha(I^{(m)}_\X)}{m}\ge \frac{ \alpha(I^{(m)}_\Y)}{m}=2=  \frac{3+2-1}{2}=\frac{\alpha(I_\X)+n-1}{n}.
$$
\end{rem}

\section{Appendix.}\label{appendice}

We recall the notation for the proofs of theorems about  standard $\k$-configurations of type $(1,b,c)$, summarized in Figure \ref{figuragenerale2}.

\begin{figure}[ht] 
\centering
\vskip -5.5cm
\psset{xunit=1cm,yunit=1cm,algebraic=true,dimen=middle,dotstyle=o,dotsize=7pt 0,linewidth=.5pt,arrowsize=3pt 2,arrowinset=0.25}
\begin{pspicture*}(-10,-8.8)(5,8.8)

\psline[linewidth=1pt](-6.7,1)(1.7,1)

\psline[linewidth=1pt](-6.7,2)(0.7,2)

\psline[linewidth=1pt](-6,3.5)(-6,0.65)

\psline[linewidth=1pt](-6.23,3.46)(-4.81,0.65)

\psline[linewidth=1pt](-6.41,3.37)(-3.65,0.66)

\psline[linewidth=1pt](-6.41,3.25)(-2.45,0.66)

\psline[linewidth=1pt](-6.51,3.2)(-1.25,0.66)

\psline[linewidth=1pt](-6.51,3.1)(1.25,1.66)

\rput[tl](-0.65,2.17){\Huge $\cdots$}
\rput[tl](0.0,1.17){\Huge$\cdots$}

\rput[tl](-7.15,1.1){\tiny$\La_1$}
\rput[tl](-7.15,2.1){\tiny$\La_2$}

\rput[tl](-6.40,2.9){\tiny$R$}
\rput[tl](-6.40,2.3){\tiny$Q_1$}
\rput[tl](-5.48,2.3){\tiny$Q_2$}
\rput[tl](-0.00,2.3){\tiny$Q_b$}

\rput[tl](-6.40,1.3){\tiny$P_1$}
\rput[tl](-5.48,1.3){\tiny$P_2$}
\rput[tl](-4.55,1.3){\tiny$P_3$}
\rput[tl](-3.65,1.3){\tiny$P_4$}
\rput[tl](0.7,1.3){\tiny$P_c$}

\rput[tl](-6.27,0.6){\tiny$\Ma_1$}
\rput[tl](-4.9,0.6){\tiny$\Na_1$}
\rput[tl](-3.7,0.6){\tiny$\Ma_2$}
\rput[tl](-2.6,0.6){\tiny$\Na_2$}
\rput[tl](-1.5,0.6){\tiny$\Ma_3$}
\rput[tl](1.3,1.7){\tiny$\Ta_i$}

\begin{scriptsize}
\psdots[dotstyle=*](-6,3)
\psdots[dotstyle=*](-6,2)
\psdots[dotstyle=*](-5,2)
\psdots[dotstyle=*](-4,2)
\psdots[dotstyle=*](-3,2)
\psdots[dotstyle=*](-6,1)
\psdots[dotstyle=*](-5,1)
\psdots[dotstyle=*](-4,1)
\psdots[dotstyle=*](-3,1)
\psdots[dotstyle=*](-2,1)
\end{scriptsize}
\end{pspicture*}
\vskip -9.4cm
\caption{A standard $\k$-configuration of type $(1,b,c)$}\label{figuragenerale2}
\end{figure}
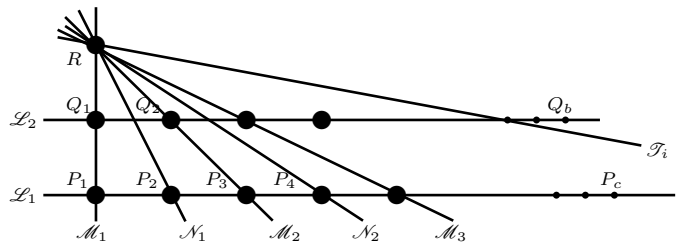 

We denote by  
$$
\begin{array}{llllllllllllllllll}
\La_1 & \text{be the line through  $P_1,P_2,\dots,P_{c}$;} \\
\La_2 & \text{be the line through  $Q_1,Q_2,\dots,Q_{b}$;} \\
\Ma_1 & \text{be the line through  $P_1,Q_1,R$;} \\
\Ma_2 & \text{be the line through  $P_3,Q_2,R$;} \\
   & \hskip 2 cm \vdots \\
  \Ma_i& \text{be the line through  $P_{2i-1},Q_i,R$,
       for $i\leq b$ and $2i\leq c+1$;}\\
\Na_1 & \text{be the line through  $P_2,R$;} \\
\Na_2 & \text{be the line through  $P_4,R$;} \\
      & \hskip 2 cm \vdots \\
 \Na_i  & \text{be the line through  $P_{2i},R$, \ for  $2i\leq c$;}\\

 \Ta_i &  \text{be the line through  $Q_{i},R$,  for $i\leq b$ and $2i\geq c+2$.}\\
\end{array}
$$
\begin{proof} [Proof of Theorem \ref{T:20210606-402}] 
Let $\X$ be a { standard} $\k$-configuration  of type $(1,b,b+1)$. If $b \geq4$ is an even  integer, we  show  that
$$
\widehat\alpha(I_\X)=\frac{9b-4}{3b}.
$$

Let
$$
\Fa=\frac{3b-2}{2}\La_1 +\frac{3b-2}{2}\La_2 +\Ma_1+\cdots+\Ma_{\frac{b}{2}+1}+\Na_1+\cdots+\Na_{\frac{b}{2}}+\Ta_{\frac{b}{2}+2}+  \cdots+\Ta_{b},
$$
so $m\Fa$ is a curve in the linear system $[I_{\frac{3b}{2}m \X}]_{\frac{9b-4}{2}m }$.
Now we need to prove that, for each $m>0$,  $\dim [I_{\frac{3b}{2}m \X}]_{\frac{9b-4}{2}m -1 }=0$.

By    Lemma \ref {L:20210429-201}, if $\dim [I_{\frac{3b}{2}m \X}]_{\frac{9b-4}{2}m -1 }\neq \{0\}$, then $\La_1$ is a fixed component of multiplicity at least 
\begin{equation}\label{M-b+1-1} 
\left\lceil \frac{\frac{3b}{2}m (b+1)-\frac{9b-4}{2}m +1}{b}\right\rceil=
\left\lceil \frac{(3b^2-6b+4)m +2}{2b}\right\rceil \ge {\frac{3b-6}{2} m} ,
\end{equation}
for the plane curves of the linear system
$[I_{\frac{3b}{2}m \X}]_{\frac{9b-4}{2}m -1 }$.

If we remove $\frac{3b-6}{2} m \La_1$, we get that $\La_2$ is a fixed component of multiplicity at least

\begin{equation}\label{M-b+1-2} 
\left\lceil \frac{\frac{3b}{2} m b- ( 3b+1 )m +1}{b-1}\right\rceil=
\left\lceil \frac{(3b^2-6b-2)m +2}{2(b-1)}\right\rceil \ge {\frac{3b-6}{2}m} .
\end{equation}
By removing $\frac{3b-6}{2} m \La_2$, we have that $\La_1$ is a fixed component of multiplicity at least 
\begin{equation}\label{M-b+1-3} 
\left\lceil \frac{3m (b+1)- \frac{3b+8}{2}m +1}{b}\right\rceil=
m + \left\lceil \frac{(b-2)m +2}{2b}\right\rceil.
\end{equation}
After removing $m\La_1$,  then $\La_2$ is a fixed component of multiplicity at least 
\begin{equation}\label{M-b+1-4} 
\left\lceil \frac{3bm- \frac{3b+6}{2}m +1}{b-1}\right\rceil=
m + \left\lceil \frac{(b-4)m +2}{2b-2}\right\rceil.
\end{equation}
Remove $\ m \La_2$.  The  residual scheme is

$$\Y= \X- \Big (\frac{3b-6}{2} m +m\Big)\La_1- \Big (\frac{3b-6}{2} m+m\Big ) \La_2 
= \frac{3b}{2} m R + \sum_i 2m P_i +\sum_i 2m Q_i,
$$
and  
$$\dim [I_{\frac{3b}{2}m \X}]_{\frac{9b-4}{2}m -1 }=
\dim [I_{\Y}]_{\frac{3b+4}{2}m -1 }.
$$
Now, by Bezout's Theorem,  the lines $\Ma_i$,  $\Na_i$, and $\Ta_i$ are fixed components.

Set 
\begin {equation} \label {barm-per-b+1}  \bar m =\min \{ m \ | \ [I_{\frac{3b}{2}m \X}]_{\frac{9b-4}{2}m -1 }\neq\{0\} \}.
\end {equation}
First observe that $\bar m \neq 1$, in fact for $m=1$,
by \eqref{M-b+1-1}, \eqref{M-b+1-2}, \eqref{M-b+1-3},  \eqref{M-b+1-4}, and
 using also the ceiling parts, we get that $\Fa$ is a curve of the linear system
$[I_{\frac{3b}{2} \X}]_{\frac{9b-4}{2} -1 }$, but $\Fa$ has degree $ \frac{9b-4}{2}$, a contradiction.

So $\bar m>1$. By by \eqref{M-b+1-1}, \eqref{M-b+1-2}, \eqref{M-b+1-3},  \eqref{M-b+1-4}, we get that
$\Fa$ is a fixed component for the linear system
$\big[I_{\frac{3b}{2} \bar m\X}\big]_{\frac{9b-4}{2}\bar m -1 }$, hence, by recalling that 
$\deg \Fa=\frac{9b-4}{2}$ and $\Fa$  contains each point of   $\X$ with  multiplicity $\frac{3b}{2}$, we get
$$\dim \big[I_{\frac{3b}{2} \bar m\X}\big]_{\frac{9b-4}{2}\bar m -1 }=
\dim \big[I_{\frac{3b}{2} \bar m\X- \Fa}\big]_{\frac{9b-4}{2}\bar m -1- \frac {9b -4}{2}}=
\dim \big[I_{\frac{3b}{2} (\bar m-1)\X}\big]_{\frac{9b-4}{2}(\bar m -1)-1},
$$
which is zero by \eqref {barm-per-b+1}, a contradiction.
\end{proof}

\begin{proof} [Proof of Theorem \ref{T:20210717}] 
Let $\X$ be a standard $\k$-configuration  of type $(1,b,2b-1)$. We show that
$$
\widehat \alpha(I_\X)=\frac{6b^2-8b+1}{2b^2-2b}.
$$
Let
$$
\begin{array}{lllllllllllllllll}
\Ca_i & \text{be the  irreducible  curve of degree $(b-1)$ through $P_2,P_4,\dots, P_{2b-2},Q_1,\dots,\widehat Q_i,\dots, Q_b, (b-2)R$} \\
& \text{for $1\le i\le b$ (see Lemma~\ref{T:20210527-109}),} 
\end{array} 
$$
and let  $\Fa$ be the following curve of degree $6b^2-8b+1$  with multiplicity $2b^2-2b$ at each point of $\X$. 
$$
\Fa=(2b^2-3b)\La_1+(2b^2-4b+1)\La_2+b\Ma_1+\cdots+b\Ma_{b}+\Ca_1+\cdots+\Ca_{b}.
$$
 Hence for $m>0$
$$
m\Fa\in [I_{(2b^2-2b)m\X}]_{(6b^2-8b+1)m}. 
$$
We  should now prove that for $m >0$,
$$
 [I_{(2b^2-2b)m\X}]_{(6b^2-8b+1)m-1}= \{0\}.
$$
Since the proof is analogous to the one of Theorem \ref{M2b-2}, assuming that the ideals which we will consider are different from zero, 
we just show  the computation that,  from Lemma \ref {L:20210429-201}, gives how many times each component of $\Fa$ is a fixed component for the curves of the linear system
$[I_{(2b^2-2b)m\X}]_{(6b^2-8b+1)m-1}$.

We get that  $\La_1$ is fixed component  of multiplicity at least 
\begin{equation}\label{M2b-1-1} 
 \left\lceil \frac{(2b-1)(2b^2-2b)m-(6b^2-8b+1)m+1}{2b-2} \right\rceil 
\geq  (2b^2-4b+1)m.
\end{equation}
By removing $(2b^2-4b+1)m\La_1$, we get that $\La_2$  is fixed component  of multiplicity at least 
\begin{equation}\label{M2b-1-2} 
 \left\lceil \frac{b(2b^2-2b)m-(4b^2-4b)m+1}{b-1} \right\rceil 
= (2b^2-4b)m +\left\lceil \frac{1}{b-1} \right\rceil.
\end{equation}
By removing $(2b^2-4b)m\La_2$,
we find
 that  $\La_1$ is fixed component  of multiplicity at least 
\begin{equation}\label{M2b-1-3} 
 \left\lceil \frac{(2b-1)^2m-2b^2m+1 }{2b-2} \right\rceil 
=  (b-2)m +  \left\lceil \frac{(2b-3)m+1}{2b-2}\right\rceil .
\end{equation}
Now we remove $(b-2)m\La_1$ and we find
 that  each $\Ma_i$ is fixed component  of multiplicity at least 
\begin{equation}\label{M2b-1-4} 
 \left\lceil \frac{(2b^2-2b)m+2bm+(b+1)m-(2b^2-b+2)m+1 }{2} \right\rceil 
=  (b-1)m+\left\lceil \frac{m+1}{2} \right\rceil .
\end{equation}
So, after we remove $((2b^2-4b+1)+(b-2))m\La_1 +(2b^2-4b)m\La_2+
\sum _{i=1}^b(b-1)m \Ma_i$, the residual scheme is 
$$\Y=(b^2-b)R + \sum _{i=1}^b (b+1)Q_i + \sum_{\text {for $i $ odd}}2mP_i
+ \sum_{\text {for $i $ even }}(b+1)mP_i ,
$$
and the degree we have to consider is
$((6b^2-8b+1)-(2b^2-4b+1)-(b-2)-(2b^2-4b)-b(b-1) )m-1=(b^2+2)m-1,
$
thus
$$ \dim [I_{(2b^2-2b)m\X}]_{(6b^2-8b+1)m-1}=
\dim [I_{\Y}]_{(b^2+2)m-1}.
$$
Now if $\Ha$ is a curve of the linear system $[I_{\Y}]_{(b^2+2)m-1}$,
the multiplicity of intersection between each $\Ca_i$ and $\Ha$ is at least 
$$|\Ca_i \cdot \Ha | \geq 
(b-2)(b^2-b)m  
+(b+1) (b-1)m+(b+1) (b-1)m=
(b^3-b^2+2b-2)m,
$$
and this number is bigger than the product of the degree of $\Ca_i$ and $\Ha$, which is
$$\deg \Ca_i \cdot \deg \Ha=(b-1)((b^2+2)m-1)=(b^3 -b^2+2b-2)m -(b-1). $$
Hence, by B\'ezout's Theorem, each curve $\Ca_i$ is a fixed component for the curves of  $[I_{\Y}]_{(b^2+2)m-1}.$

Now let 
\begin {equation} \label {barm-per-2b-1} 
 \bar m =\min \{ m |\ [I_{(2b^2-2b)m\X}]_{(6b^2-8b+1)m-1}\neq\{0\} \}.
\end {equation}
We have $\bar m>1$, in fact for $m=1$,
 by \eqref{M2b-1-1}, \eqref{M2b-1-2}, \eqref{M2b-1-3},  \eqref{M2b-1-4}, and
 using also the ceiling parts, we get that $\Fa$ should be a curve in the linear system
$[I_{(2b^2-2b)\X}]_{6b^2-8b}$, but $\Fa$ has degree $ 6b^2-8b+1$, a contradiction.

Hence $\bar m>1$. 

By  the above computation, 
then
$\Fa$ is a fixed component for the linear system
$ [I_{(2b^2-2b)m\X}]_{(6b^2-8b+1)m-1}.$
We have
$$
\begin{array}{llllllllllllll}
\dim   [I_{(2b^2-2b)\bar m\X}]_{(6b^2-8b+1)\bar m-1}&= &
 \dim  [I_{(2b^2-2b)\bar m\X-\Fa}]_{(6b^2-8b+1)\bar m-1-(6b^2-8b+1)} \\
&= &
 \dim  [I_{(2b^2-2b)(\bar m-1)\X}]_{(6b^2-8b+1)(\bar m-1)-1}
\end{array}
$$
which is zero by \eqref {barm-per-2b-1}, a contradiction.
\end{proof}

\begin{proof} [Proof of Theorem \ref{T:20210428-203}] 
Let $\X$ be a  standard $\k$-configuration  of type $(1,b,2b)$. We show that
$$
\widehat\alpha (I_\X)=\frac{6b-5}{2b-1}.
$$
Let
$$
\begin{array}{lllllllllllllllll}
\Ca& \text{be the irreducible curve of degree $b$ through $P_2,P_4,\dots, P_{2b},Q_1,\dots, Q_b, (b-1)R$} \\
& \text{  (see Lemma~\ref{T:20210527-109}),} 
\end{array} 
$$
and let $\Fa$ be the following curve of degree $(6b-5)$  with multiplicity exactly  $(2b-1)$ at the points of~$\X$, 
$$
\Fa=(2b-2)\La_1+(2b-3)\La_2+\Ma_1+\cdots+\Ma_b+\Ca.
$$
Hence, for $m>0$, 
$
m  \Fa\in [I_{m (2b-1)\X}]_{m (6b-5)}.
$
Now we will show that for each $m>0$ we have
$$
  [I_{m (2b-1)\X}]_{m (6b-5)-1}=\{0\},
$$
and the conclusion will  follow from Lemma \ref {L:20210722-1}.

Assume that $ [I_{m (2b-1)\X}]_{m (6b-5)-1}\neq\{0\}$ for some $m>0$.

Let $\Ha$ be a curve of the linear system $[I_{m(2b-1)\X}]_{m(6b-5)-1}$. Then
the multiplicity of intersection between $\Ca$ and $\Ha$ is at least $(2b-1)m$ in each of the points $P_i$ and $Q_i$ and at least $(b-1)(2b-1)m$ in $R$. Since we have $2b$ points $P_i$ and $Q_i$,
$$|\Ca \cdot \Ha | \geq 2b(2b-1)m+(b-1)(2b-1)m,$$
and this number is bigger than the product of the degree of $\Ca$ and $\Ha$, which is
$b (m (6b-5)-1)$, in fact
$$2b(2b-1)m+(b-1)(2b-1)m - b (m (6b-5)-1)=
m+b >0.
$$
Hence, by B\'ezout's Theorem, the curve $\Ca$ is a fixed component for the curves of  $[I_{m(2b-1)\X}]_{m(6b-5)-1}$.

Moreover, for the curves of  this linear system, by Lemma \ref {L:20210429-201},  $\Ma_i$, ($1\le i \le b$),  is a fixed component of multiplicity at least 
$$
\left\lceil \frac{3(2b-1)m -(6b-5)m +1}{2} \right\rceil
=\left\lceil \frac{2m +1}{2} \right\rceil = m+1,
$$
and  $\La_1$ is a fixed component of multiplicity at least 
$$
 \left\lceil \frac{2b(2b-1)m -(6b-5)m +1}{2b-1} \right\rceil
 =   (2b-3)m +  \left\lceil \frac{2m+1}{2b-1} \right\rceil.
$$
If we remove the curve $(2b-3)m\La_1$ we get
$$\dim [I_{m(2b-1)\X}]_{m(6b-5)-1}=
\dim [I_{m(2b-1)\X -  (2b-3)m\La_1})]_{(4b-2)m-1}.
$$
If this dimension is different from zero, by Lemma \ref {L:20210429-201},  we get that $\La_2$ is a fixed component of multiplicity at least 
$$
 \left\lceil \frac{(2b-1)m\cdot b-(4b-2)m+1}{b-1} \right\rceil
=   (2b-4)m+ \left\lceil\frac{(b-2)m+1}{b-1} \right\rceil 
$$ 
for the curves of  $[I_{m(2b-1)\X}]_{m(6b-5)-1}$.

Now let 
\begin {equation} \label {barm-per-2b}  \bar m =\min \{ m |[I_{m (2b-1)\X}]_{m (6b-5)-1}\neq\{0\} \}.
\end {equation}
We have $\bar m>1$, in fact for $m=1$, by the computation above, the curve $\Fa$of degree $6b-5$
should be a fixed component for the linear system,  $[I_{(2b-1)\X}]_{6b-4}$,
a contradiction.

Hence $\bar m>1$. Since
$\Fa$ is a fixed component for the linear system
$[I_{m(2b-1)\X}]_{m(6b-5)-1}$
 we have
$$\dim  [I_{\bar m(2b-1)\X}]_{\bar m (6b-5)-1}= \dim [I_{\bar m (2b-1)\X-\Fa}]_{\bar m(6b-5)-1-(6b-5)} =\dim [I_{(\bar m -1)(2b-1)\X}]_{(\bar m-1)(6b-5)-1} ,
$$
which is zero by (\ref {barm-per-2b} ), a contradiction.
\end{proof}

\begin{proof} [Proof of Theorem \ref{T:20210618-203}] 
Let $\X$ be a  standard $\k$-configuration  of type $(1,b,2b+1)$. We show that
$$
\widehat\alpha (I_\X)=\frac{6b^2-2b-3}{2b^2-1}.
$$
Let
$$
\begin{array}{lllllllllllllllll}
\Ca_i & \text{be the irreducible curve of degree $b$ through  $P_2,P_4,\dots,\widehat P_{2i},\dots, P_{2b}, P_{2b+1}$ ,} \\
 & \text{$Q_1,\dots,Q_b, (b-1)R$ for $1\le i\le b$,} \\
\Ca_{b+1}& \text{be the irreducible curve  of degree $b$ through $P_2,P_4,\dots, P_{2b}, Q_1,\dots,Q_b, (b-1)R$;} 
\end{array} 
$$
(see Lemma~\ref{T:20210527-109} for the $b+1$ curves $\Ca_i$).
Note that the  curve $\Ca_1+ \cdots +\Ca_{b+1}$ has degree $b(b+1)$,  multiplicity $b+1$ at each of the points $Q_1,\dots,Q_b$,  multiplicity $b$ at each of the points $P_2,P_4,\dots, P_{2b}, P_{2b+1}$, and multiplicity $b^2-1$ at $R$.
Let
$$
\Fa=(2b^2-b-1)\La_1+(2b^2-2b-2))\La_2+b\Ma_1+\cdots+b\Ma_b+\Ca_1+\cdots+\Ca_{b+1}.
$$
Then $\Fa$ is a curve of degree $(6b^2-2b-3)$ with  multiplicity $(2b^2-1)$ at each point of $\X$. Hence for $m>0$
$$
m\Fa\in [I_{(2b^2-1)m\X}]_{(6b^2-2b-3)m}. 
$$
We now have to prove that
$$
[ I_{(2b^2-1)m\X}]_{(6b^2-2b-3)m-1}=0.
$$
Assume that for some $m >0$, 
$[ I_{(2b^2-1)m\X}]_{(6b^2-2b-3)m-1}\neq \{0\}.$

Analogously to the proof of Theorem \ref{T:20210428-203}, let $\Ha$ be a curve of the linear system $[I_{(2b^2-1)m\X}]_{(6b^2-2b-3)m-1}$.  Then
the multiplicity of intersection between each $\Ca_i$ and $\Ha$ is at least 
$(2b^2-1)m$ in each of the $2b$ points $P_i$ and $Q_i$ and at least $(b-1)(2b^2-1)m$ in $R$, so, 
$$|\Ca_i \cdot \Ha | \geq 2b(2b^2-1)m+(b-1)(2b^2-1)m,$$
and this number is bigger than the product of the degree of $\Ca_i$ and $\Ha$, which is
$b ( (6b^2-2b-3)m-1)$.
Hence, by B\'ezout's Theorem, each curve $\Ca_i$ is a fixed component for the curves of  $\ [I_{(2b^2-1)m\X}]_{(6b^2-2b-3)m-1}$.

Moreover, for the curves of  this linear system, by Lemma \ref {L:20210429-201},  
each $\Ma_i$ is a fixed component of multiplicity at least 
$$
 \left\lceil \frac{3(2b^2-1) m-(6b^2-2b-3)m+1}{2} \right\rceil
= bm+1,
$$
 $\La_1$  is a fixed component of multiplicity at least 
$$
 \left\lceil \frac{(2b^2-1) (2b+1)m-(6b^2-2b-3)m+1}{2b} \right\rceil=
 \left\lceil \frac{(4b^3-4b^2+2) m+1 }{2b} \right\rceil
=  (2b^2-2b)m+\left\lceil \frac{2 m+1}{2b} \right\rceil,
$$
and, by removing $ (2b^2-2b)m\La_1$, we get that   $\La_2$ is a fixed component of multiplicity at least 
$$
 \left\lceil \frac{(2b^2-1)m\cdot b-(4b^2-3)m+1}{b-1} \right\rceil 
=  (2b^2-2b-3)m + \left\lceil \frac{1}{b-1} \right\rceil.
$$
Now let 
\begin {equation} \label {barm-per-2b+1}  \bar m =\min \{ m |[I_{m (2b^2-1)\X}]_{m (6b^2-2b-3)-1}\neq\{0\} \}.
\end {equation}
We have $\bar m>1$, in fact for $m=1$, by the computation above, the curve $\Fa'$of degree $6b^2-3b-1$,
$$\Fa'=(2b^2-2b+1)\La_1+(2b^2-2b-2 )\La_2+b\Ma_1+\cdots+b\Ma_b+\Ca_1+\cdots+\Ca_{b+1},
$$
should be a fixed component for the linear system, so
$$
\begin{array}{lllllll}
\dim [ I_{(2b^2-1)\X}]_{(6b^2-2b-3)-1} &=&
\dim [ I_{(2b^2-1)\X-\Fa'}]_{(6b^2-2b-4)- (6b^2-3b-1)     } \\
&=&
\dim [ I_{(b-2)P_1+\cdots+(b-2)P_{2b+1}}]_{(b-3)}\\
&=&0,
\end{array}
$$
a contradiction.

Hence $\bar m>1$. By  the computation above
$\Fa$ is a fixed component for  $[ I_{(2b^2-1)\bar m\X}]_{(6b^2-2b-3)\bar m-1}$,
hence we have
$$
\begin{array}{lllllll}
\dim [ I_{(2b^2-1)\bar m\X}]_{(6b^2-2b-3)\bar m-1}  &=&
\dim [ I_{(2b^2-1)\bar m\X-\Fa}]_{(6b^2-2b-3)\bar m-1- (6b^2-2b-3) } \\
&=&\dim [ I_{(2b^2-1)(\bar m-1)\X}]_{(6b^2-2b-3)(\bar m-1)-1 },
\end{array}
$$
which is zero by \eqref {barm-per-2b+1}, a contradiction.
\end{proof}

\end{document}